\documentclass{article}
\pdfoutput=1 

\usepackage[margin=2.5cm]{geometry}
\usepackage{bbm}
\usepackage[utf8]{inputenc} 
\usepackage[T1]{fontenc}    
\usepackage{hyperref}       
\usepackage{url}            
\usepackage{booktabs}       
\usepackage{amsfonts,amsmath,amssymb, amsthm}       
\usepackage{nicefrac,verbatim,bm}       
\usepackage{microtype}      
\usepackage{graphicx}
\usepackage{color}
\usepackage[numbers]{natbib}

\newcommand{\bas}[1]{\begin{align*}#1\end{align*}}
\newcommand{\ba}[1]{\begin{align}#1\end{align}}
\newcommand{\spn}{\mathrm{span}}
\newcommand{\ip}[1]{\langle#1\rangle}
\newcommand{\bone}{\mathbf{1}}

\newcommand{\bzero}{\mathbf{0}}

\newcommand{\rd}{\color{red}}
\newcommand{\bk}{\color{black}}

\newcommand{\cC}{\mathcal{C}}

\newcommand{\E}{\mathbb{E}}
\newcommand{\icone}{\bone_{\mathcal{C}_1}}
\newcommand{\ictwo}{\bone_{\mathcal{C}_2}}
\newcommand{\icthree}{\bone_{\mathcal{C}_3}}
\newcommand{\ind}{\mathbbm{1}}

\renewcommand{\P}{\mathbb{P}}
\newcommand{\hpi}{\tilde{\pi}}
\newcommand{\pstar}{p_0}
\newcommand{\qstar}{q_0}
\newcommand{\tstar}{t_0}
\newcommand{\lambdastar}{\lambda_0}
\newcommand{\astar}{a}
\newcommand{\phat}{p^{(1)}}
\newcommand{\qhat}{q^{(1)}}

\newcommand{\var}{\text{Var}}

\newtheorem{theorem}{Theorem}
\newtheorem{proposition}[theorem]{Proposition}
\newtheorem{lemma}[theorem]{Lemma}
\newtheorem{corollary}[theorem]{Corollary}
\newtheorem{remark}{Remark}

\title{When random initializations help: a study of variational inference for community detection }

\author{
  Purnamrita Sarkar\thanks{Equal contribution.} \thanks{All authors contributed equally to the short version of the paper that appeared in NeurIPS 2018.} \\
  Department of Statistics and Data Science\\
    University of Texas, Austin\\
      Austin, TX 78712, USA \\
  \texttt{purna.sarkar@austin.utexas.edu} \\
   \and
 Y. X. Rachel Wang\footnotemark[1]
       \footnotemark[2]\\
  School of Mathematics and Statistics\\
       University of Sydney\\
       NSW 2006, Australia \\
  \texttt{rachel.wang@sydney.edu.au} \\
  \and
  Soumendu Sundar Mukherjee\footnotemark[2]\\
  Interdisciplinary Statistical Research Unit (ISRU)\\
       Indian Statistical Institute, Kolkata\\
       Kolkata 700108, India \\
    \texttt{soumendu041@gmail.com}
}
\date{}

\begin{document}
\maketitle

\begin{abstract}
Variational approximation has been widely used in large-scale Bayesian inference recently, the simplest kind of which involves imposing a mean field assumption to approximate complicated latent structures. Despite the computational scalability of mean field, theoretical studies of its loss function surface and the convergence behavior of iterative updates for optimizing the loss are far from complete. In this paper, we focus on the problem of community detection for a simple two-class Stochastic Blockmodel (SBM) with equal class sizes. Using batch co-ordinate ascent (BCAVI) for updates, we show different convergence behavior with respect to different initializations. When the parameters are known or estimated within a reasonable range and held fixed, we characterize conditions under which an initialization can converge to the ground truth. On the other hand, when the parameters need to be estimated iteratively, a random initialization will converge to an uninformative local optimum.
\end{abstract}

\section{Introduction}\label{sec:intro}
Variational approximation has recently gained a huge momentum in contemporary Bayesian statistics~\citep{Jordan:1999:VM,Blei:2003:LDA,Jaakkola:1999:IMF:308574.308663}. Mean field is the simplest type of variational approximation, and is a popular tool in large scale Bayesian inference. It is particularly useful for problems which involve complicated latent structure, so that direct computation with the likelihood is not feasible. The main idea of variational approximation is to obtain a tractable lower bound on the complete log-likelihood of any model. This is, in fact, akin to the Expectation Maximization algorithm~\citep{dempster1977maximum}, where one obtains a lower bound on the marginal log-likelihood function via the expectation with respect to the conditional distribution of the latent variables under the current estimates of the underlying parameters.  In contrast, for  mean field variational approximation, the lower bound or ELBO is computed using the expectation with respect to a product distribution over the latent variables. 

While there are many advances in developing new mean field type approximation methods for Bayesian models, the theoretical behavior of these algorithms is not well understood. There is one line of theoretical work that studies the asymptotic consistency of variational inference, most of which  focuses on the global optimizer of variational methods under specific models. For example, for Latent Dirichlet Allocation (LDA)~\citep{Blei:2003:LDA} and Gaussian mixture models, it is shown in~\cite{pati2017statistical} that the global optimizer is statistically consistent. \cite{westling2015beyond} connects variational estimators to profile M-estimation, and shows consistency and asymptotic normality of those estimators.
For Stochastic Blockmodels (SBM)~\citep{holland1983stochastic,hofman2008bayesian}, \cite{bickel2013asymptotic} shows that the global optimizer of the variational log-likelihood is consistent and asymptotically normal. 
For more general cases, \cite{wang2017frequentist} proves a variational Bernstein-von Mises theorem, which states that the variational posterior converges to the Kullback-Leibler minimizer of a normal distribution, centered at the truth.

Recently, a lot more effort is being directed towards understanding the statistical convergence behavior of non-convex algorithms in general. For Gaussian mixture models and exponential families with missing data, \cite{wang2004convergence,wang2006convergence} prove local convergence to the true parameters. The same authors also show that the covariance matrix from variational Bayesian approximation for the Gaussian mixture model is ``too small'' compared with that obtained for the maximum likelihood estimator \citep{wang2005inadequacy}. The robustness of variational Bayes estimators is further discussed in~\cite{giordano2017covariances}. For LDA, \cite{awasthi2015some} shows that, with proper initialization, variational inference algorithms converge to the global optimum.

In this paper, we will focus on the community detection problem in networks under SBM. Here the latent structure involves unknown community memberships and as a result, the data likelihood requires summing over all possible community labels. Optimization of the likelihood involves a combinatorial search, and thus is infeasible for large-scale graphs. The mean field approximation has been used popularly for this task \citep{blei2017variational,zhang2017convergence}. In~\cite{bickel2013asymptotic}, it is proved that the global optimum of the mean field approximation to the likelihood behaves optimally in the dense degree regime, where the average expected degree of the network grows faster than the logarithm of the number of vertices. In~\cite{zhang2017convergence}, it is shown that if the initialization of mean field is close enough to the truth then one gets convergence to the truth at the minimax rate.  However, in practice, it is usually not possible to initialize like that unless one uses a pilot algorithm. Most initialization techniques like spectral clustering~\citep{rohe2011spectral,ng2002spectral} will return correct clustering in the dense degree regime, thus rendering the need for mean field updates redundant. 

Indeed, in many practical scenarios, without prior knowledge one simply uses multiple random initializations,  the efficacy of which is model-dependent. In order to understand the behavior of random initializations, one needs to first better understand the landscape of the mean field loss. There are few such studies for non-convex optimization in the literature; notable examples include~\cite{mei:nonconvex,ghorbani2018instability,jin2016local,xu2016global}. In~\cite{xu2016global}, the authors fully characterize the landscape of the likelihood of the equal proportion Gaussian Mixture Model with two components, where the main message is that most random initializations should indeed converge to the ground truth. In contrast, for topic models, it has been established that, for some parameter regimes, variational inference exhibits instability and returns a posterior mean that is uncorrelated with the truth~\cite{ghorbani2018instability}.  In this respect, for network models, there has not been much work characterizing the behavior of the variational loss surface.

 In this article, in the context of a specific SBM, we give a complete characterization of all the critical points and establish the behavior of random initializations 
for batch co-ordinate ascent (BCAVI) updates for mean field likelihood (with known and unknown model parameters).
Our results thus complement those of~\cite{zhang2017theoretical}. For simplicity, we work with equal-sized two-class stochastic blockmodels. When the parameters are known, we show conditions under which random initializations can converge to the ground truth. In particular, we show that centering random initializations around a half ensures convergence happens a good fraction of time, and this property holds even if we only have access to reasonable estimates of true parameters. We also analyze the setting with unknown model parameters, where they are estimated jointly with the community memberships. In this case, we see that indeed, with high probability, a random initialization never converges to the ground truth, thus showing the critical importance of a good initialization for network models.
\section{Setup and preliminaries}\label{sec:setup}
The stochastic blockmodel \texttt{SBM}$(B,Z,\pi)$ is a generative model of networks with community structure on $n$ nodes. Its dynamics is as follows: there are $K$ communities $\{1,\ldots,K\}$ and each node belongs to a single community, where this membership is captured by the rows of the $n \times K$ matrix $Z$, where the $i$th row of $Z$, i.e. $Z_{i,\cdot}$, is the community membership vector of the $i$th node and has a \texttt{Multinomial}$(1;\pi)$ distribution, independently of the other rows. Given the community structure, links between pairs of nodes are determined solely by the block memberships of the nodes in an independent manner. That is, if $A$ denotes the adjacency matrix of the network, then given $Z$, $A_{ij}$ and $A_{kl}$ are independent for $(i,j)\ne (k,l)$, $i<j$, $k<l$, and
\[
\P(A_{ij} = 1 \mid Z) = \P(A_{ij} = 1 \mid Z_{ia} = 1, Z_{jb} = 1) = B_{ab}.
\]
$B=((B_{ab}))$ is called the block (or community) probability matrix. We have the natural restriction that $B$ is symmetric for undirected networks.

The block memberships are hidden variables and one only observes the network in practice. The goal often is to fit an appropriate SBM to learn the community structure, if any, and also estimate the parameters $B$ and $\pi$.

The complete likelihood for the SBM is given by
\begin{equation}
\P(A, Z; B,\pi) = \prod_{i<j}\prod_{a, b} (B_{ab}^{A_{ij}} (1 - B_{ab})^{1 - A_{ij}})^{Z_{ia}Z_{jb}} \prod_i\prod_a \pi_a^{Z_{ia}}.
\end{equation}
As $Z$ is not observable, if we integrate out $Z$, we get the data likelihood
\begin{equation}
\P(A;B,\pi) = \sum_{Z\in\mathcal{Z}} \P(A, Z; B,\pi),
\end{equation}
where $\mathcal{Z}$ is the space of all $n \times K$ matrices with exactly one $1$ in each row.

In principle we can optimize the data likelihood to estimate $B$ and $\pi$. However, $\P(A;B,\pi)$ involves a sum over a complicated large finite set (the cardinality of this set is $K^n$), and hence is not easy to deal with.
A well-known alternative approach is to  optimize the variational log-likelihood \citep{bickel2013asymptotic}, which has a less complicated dependency structure, the simplest of which is mean field log-likelihood (see, e.g., \citep{wainwright2008graphical}). We defer a detailed discussion of the mean field principle in the Appendix. 

For the SBM, the variational log-likelihood with respect to a distribution $\psi$ is given by
\[
	\sum_Z \log \bigg(\frac{\P(A,Z; B,\pi)}{\psi(Z)} \bigg)\psi(Z) = \E_\psi \bigg( \sum_{i < j, a, b}Z_{ia} Z_{jb}(\theta_{ab} A_{ij} - f(\theta_{ab})) \bigg) - \texttt{KL}(\psi || \pi^{\otimes n}),
\]
where $\theta_{ab} = \log \left(\frac{B_{ab}}{1- B_{ab}}\right),$ $f(\theta) = \log (1 + e^\theta)$ and $\pi^{\otimes n}$ denotes the product measure on $\mathcal{Z}$ with the rows of $Z$ being i.i.d. \texttt{Multinomial}$(1;\pi)$. 
A special case of the variational log-likelihood is the mean field log-likelihood (see, e.g., \citep{wainwright2008graphical}), where one approximates $\Psi$ by
\begin{equation}
\Psi_{MF} \equiv \{\psi : \psi(z_1,\ldots,z_n) = \prod_{j=1}^n \psi_j(z_j)\}.
\end{equation}
Define 
$\ell_{MF}(\psi,\theta, \pi) 
= \sum_{i < 
	j, a,b}\psi_{ia}\psi_{jb}(\theta_{ab} A_{ij} - f(\theta_{ab})) - 
\sum_{i}\texttt{KL}(\psi_i||\pi).
$
 For SBM the mean field approximation is equivalent to optimizing $\ell_{MF}(\psi, \theta, \pi)$ as follows:
\begin{align*}
&\max_\psi \ell_{MF}(\psi, \theta, \pi)\\
\mbox{subject to}
&	\sum_a \psi_{ia} = 1, \text{ for all } 1\le i \le n\\
&	\psi_{ia} \ge 0, \text{ for all } 1\le i\le n, 1\le a \le K,
\end{align*}
where each $\psi_{i}$ is a discrete probability distribution over 
$\{1,\ldots,K\}.$ 
\subsection{Mean field updates for a two-parameter two-block SBM}
Consider the stochastic blockmodel with two blocks with prior block probability $\pi, 1 - \pi$ respectively  and block probability matrix $B = (p - q)I + q J$, where $p > q$, $I$ is the identity matrix, and $J = \bone\bone^\top$ is the matrix of all 1's. For simplicity, we will denote $\psi_{i1}$ as $\psi_i$. Then the mean field log-likelihood is
\begin{align*}
\ell(\psi, p, q, \pi) &= \frac{1}{2}\sum_{i,j : i \ne j} [\psi_i (1 - \psi_j) + \psi_j (1 - \psi_i)] [A_{ij} \log\bigg(\frac{q}{1 - q}\bigg) + \log(1 - q)]\\
&\qquad+ \frac{1}{2}\sum_{i,j : i \ne j} [\psi_i \psi_j + (1 - \psi_i) (1 - \psi_j)] [A_{ij} \log\bigg(\frac{p}{1 - p}\bigg) + \log(1 - p)]\\
&\qquad\qquad- \sum_{i} [\log \bigg(\frac{\psi_i}{\pi} \bigg) \psi_i
+ \log \bigg(\frac{1 - \psi_i}{1 - \pi} \bigg) (1 - \psi_i)].
\end{align*}
For simplicity of exposition, we will assume that $\pi$ (which is essentially a prior on the block memberships) is known and equals $1/2$. Let $\mathcal{C}_i, i = 1, 2$ be the two communities. Let $\hpi = \frac{|\mathcal{C}_1|}{n}$. It is clear that $\hpi = \frac{1}{2} + O_P(\frac{1}{\sqrt{n}}).$ Assuming $\hpi = \frac{1}{2}$ from the start will not change our conclusions but make the algebra a lot nicer, which we do henceforth. Now
\begin{align*}
	\frac{\partial \ell}{\partial \psi_i} &= \frac{1}{2} \sum_{j:j \ne i} 2[1 - 2\psi_j] [A_{ij} \log\bigg(\frac{q}{1 - q}\bigg) + \log(1 - q)] \\
	&\qquad + \frac{1}{2} \sum_{j:j \ne i} 2[2\psi_j - 1] [A_{ij} \log\bigg(\frac{p}{1 - p}\bigg) + \log(1 - p)] - \log\bigg( \frac{\psi_i}{1 - \psi_i} \bigg)\\
	&= 4t \sum_{j:j \ne i} (\psi_j - \frac{1}{2})(A_{ij} - \lambda) - \log\bigg( \frac{\psi_i}{1 - \psi_i} \bigg),
\end{align*}
where $t = \frac{1}{2}\log\big(\frac{p(1 - q)}{q(1 - p)}\big)$ and $\lambda = \frac{1}{2t}\log\big(\frac{1 - q}{1 - p}\big).$ Detailed calculations of other first and second order partial derivatives are given in Section~\ref*{sec:deriv_stat_points} of the Appendix.
The co-ordinate ascent (CAVI) updates for $\psi$ are
\[
  \log\frac{\psi_i^{(new)}}{1 - \psi_i^{(new)}} =  4t \sum_{j \ne i} (\psi_j - \frac{1}{2})(A_{ij} - \lambda).
\]
Introducing an intermediate variable $\xi$ for the updates, let $f(x) = \log (\frac{x}{1 - x} )$ and $\xi_i = f(\psi_i)$. Then at iteration $s$, given the current values of $p$ and $q$ for computing $t$ and $\lambda$,
the batch version (BCAVI) of this is
\begin{align*}
  \xi^{(s)} = 4t (A - \lambda(J - I))(\psi^{(s - 1)} - \frac{1}{2}\bone),
\end{align*}
and $\psi^{(s)} = g(\xi^{(s)})$, where $g$ is the sigmoid function $g(x)=1/(1+e^{-x})$.

We will study these updates in two setttings: i) when the true model parameters $\pstar, \qstar$ are known (or estimated and kept fixed), and ii) when the model parameters $\pstar,\qstar$ need to be jointly estimated with $\psi$. The detailed BCAVI updates for each setting will be described in Section~\ref{sec:main}.



\section{Main results}\label{sec:main}
In this section, we state and discuss our main results. All the proofs appear in the Appendix.

 We begin with introducing some notations. In the following, we will see the following vectors repeatedly: $\psi = \frac{1}{2}\bone, \bone, \bzero, \icone, \ictwo.$ Among these, $\bone$ corresponds to the case where every node is assigned by $\psi$ to $\cC_1$, and, similarly, for $\bzero$, to $\cC_2.$ On the other hand, $\bone_{\cC_i}$ are the indicators of the clusters $\cC_i$ and hence correspond to the ground truth community assignment. Finally, $\frac{1}{2}\bone$ corresponds to the solution where a node belong to each community with equal probability. 
 
 The next propositions show some useful inequalities for $t$ and $\lambda$ computed from general $p$ and $q$. 
 
\begin{proposition}\label{prop:bd_t_lambda}
	Suppose $1 > p > q > 0$. Then
	\begin{enumerate}
		\item $\frac{(p - q) (1 + p - q)}{2(1 - q) p} < t < \frac{(p - q) (1 - p + q)}{2(1 - p) q},$ and
		\item $q < \lambda < p$.
	\end{enumerate}
\end{proposition}

The next proposition refines the separation between $\lambda$ and $p$, $q$, when $p\asymp q \asymp \rho_n, \rho_n \to 0$.

\begin{proposition}
	If $p\asymp q \asymp \rho_n, \rho_n \to 0$ and $p-q=\Omega(\rho_n)$, then
	\begin{align}
	\lambda-q&=\Omega(\rho_n) > 0, \label{eq:lambdalb_sep} \\
	\frac{p+q}{2}-\lambda&=\Omega(\rho_n)>0. \label{eq:lambdaub_sep}
	\end{align}
	\label{prop:lambda_lbub}
\end{proposition}

\subsection{Known \texorpdfstring{$\pstar, \qstar$}{}:}
\label{subsec:knownpq}
In this case, denoting the true model parameters $\pstar, \qstar$ ($\pstar>\qstar$), we assume these parameters are known and thus need only consider the updates for $\psi$. We consider the case where the true $\pstar$, $\qstar$ are of the same order, that is, $\pstar\asymp \qstar \asymp \rho_n$ with $\rho_n$ possibly going to 0. The BCAVI updates are:
\begin{equation}\label{eq:bcavi_sample_known_pq}
\xi^{(s+1)} = 4\tstar (A-\lambdastar(J-I))(\psi^{(s)} - \frac{1}{2}\bone),\\
\end{equation}
where $\tstar$ and $\lambdastar$ are calculated using $\pstar$ and $\qstar$.
In what follows, we will also study the population version of this update which replaces $A$ by $\E(A \mid Z) = ZBZ^\top - \pstar I =: P - \pstar I$. Hence for convenience, denote $M:=P - \pstar I- \lambdastar(J - I)$. 
The population BCAVI updates are
\begin{equation}\label{eq:bcavi_known_pq}
\xi^{(s+1)} = 4\tstar M(\psi^{(s)} - \frac{1}{2}\bone).\\
\end{equation}

The eigendecomposition of $P - \lambdastar J$ will play a crucial role in our analysis. Note that it has rank two and two eigenvalues $e_\pm = n\alpha_\pm$, where $\alpha_+ =  \frac{\pstar + \qstar}{2} - \lambdastar, \alpha_- = \frac{\pstar - \qstar}{2}$, with eigenvectors $\bone$ and $\icone - \ictwo$ respectively. Now it can be easily checked that the eigenvalues of $M$ are $\nu_1 = e_+ -  (\pstar - \lambdastar)$, $\nu_2 = e_- - (\pstar - \lambdastar)$ and $\nu_j = -  (\pstar - \lambdastar)$, $j = 3,\ldots, n$. The eigenvector of $M$ corresponding to $\nu_1$ is $u_1 = \bone$, and the one corresponding to $\nu_2$ is $u_2 = \icone - \ictwo.$

We first present a proposition related to the landscape of the objective function. In the known $\pstar, \qstar$ case, $\frac{1}{2}\bone$ is a saddle point of the population mean field log-likelihood.
\begin{proposition}\label{prop:saddle}
	$\psi = \frac{1}{2}\bone$ is a saddle point of the population mean field log-likelihood when $\pstar$ and $\qstar$ are known, for all $n$ large enough.
\end{proposition}

We next give conditions on the initialization which determine their convergence behavior when using the population BCAVI~\eqref{eq:bcavi_known_pq}. To facilitate our discussion, we will write the BCAVI updates in the eigenvector coordinates of $M$. To this end, define $\zeta^{(s)}_i = \langle \psi^{(s)}, u_i \rangle/\|u_i\|^2 = \langle \psi^{(s)}, u_i \rangle/n$, for $i = 1, 2$. We can then write
\begin{align*}
\psi^{(s)} &= \langle\psi^{(s)}, u_1/\|u_1\| \rangle  u_1/\|u_1\| + \langle\psi^{(s)}, u_2/\|u_2\| \rangle  u_2/\|u_2\| + v^{(s)} = \zeta_1^{(s)} u_1 + \zeta_2^{(s)} u_2 + v^{(s)}.
\end{align*}
So, using \eqref{eq:bcavi_known_pq} in conjunction with the above decomposition, coordinate-wise we have:
\begin{align}\label{eq:main_decomp}
\xi^{(s+1)}_i &= 4\tstar n\bigg((\zeta_1^{(s)} - \frac{1}{2})\alpha_+ +  \sigma_i \zeta_2^{(s)}\alpha_- \bigg) + 4\tstar \nu_3 \bigg( (\zeta_1^{(s)} - \frac{1}{2}) + \sigma_i \zeta_2^{(s)} + v^{(s)}_i \bigg)\\
&=: n a_{\sigma_i}^{(s)} + b_i^{(s)},
\end{align}
where $\sigma_i = 1$, if $i$ is in $\mathcal{C}_1$, and $-1$ otherwise.

\begin{theorem}[Population behavior]\label{thm:pop_known_pq} The limit behavior of the population BCAVI updates~\eqref{eq:bcavi_known_pq} is characterized by the signs of $\alpha_+$ and $a_{\pm1}^{(0)}$, where $\alpha_+=(\pstar+\qstar)/2-\lambdastar$ and $a_{\pm1}^{(s)}$ for iteration $s$ is defined in~\eqref{eq:main_decomp}. Assume that $|na_{\pm1}^{(0)}| \rightarrow \infty$. Define $\ell(\psi^{(0)}) = \ind(a_{+1}^{(0)} > 0) \icone + \ind(a_{-1}^{(0)} > 0) \ictwo$. Then, we have
	$$\frac{\|\psi^{(1)} - \ell(\psi^{(0)})\|^2}{n} = O(\exp(-\Theta(n\min\{|a^{(0)}_{+1}|, |a^{(0)}_{-1}|\}))) = o(1).$$
	We also have for any $s\ge 2$
	\begin{align*}
	\frac{\|\psi^{(s)} - \ell(\psi^{(0)})\|^2}{n} =\begin{cases} O(\exp(-\Theta(n\tstar \alpha_-))),&\mbox{if $a_{+1}^{(0)}a_{-1}^{(0)} < 0$},\\
	O(\exp(-\Theta(n\tstar|\alpha_+|)),&\mbox{if $a_{+1}^{(0)}a_{-1}^{(0)} > 0$, and $\alpha_+ > 0$}.\\
	\end{cases}
	\end{align*}

Finally, if $a_{+1}^{(0)}a_{-1}^{(0)} > 0$ and $\alpha_+ < 0$, then, for any $s \ge 2$, we have
\[
	\min \bigg\{\frac{\|\psi^{(s)} - \bone\|^2}{n}, \frac{\|\psi^{(s)} - \bzero\|^2}{n}\bigg\} = O(\exp(-\Theta(n\tstar|\alpha_+|)).
\]
In fact, in this case, $\psi^{(s)}$ cycles between $\bone$ and $\bzero$, in the sense that it is close to $\bone$ is one iteration, and to $\bzero$ in the next and so on.
\end{theorem}

\begin{remark}
	We see from Theorem~\ref{thm:pop_known_pq} that, essentially, we have exponential convergence within two iterations.
\end{remark}

Now we turn to the sample behavior of the updates in~\eqref{eq:bcavi_sample_known_pq}.
\begin{theorem}[Sample behavior]\label{thm:samp_known_pq} For all $s \geq 1$, the same conclusion as Theorem~\ref{thm:pop_known_pq} holds for the sample BCAVI updates in~\eqref{eq:bcavi_sample_known_pq} with high probability as long as $n|a^{(0)}_{\pm 1}| \gg \max \{\sqrt{n\rho_n}\|\psi^{(0)} - \frac{1}{2}\|_{\infty}, 1 \}$, $\sqrt{n\rho_n} = \Omega(\log n)$ and $\psi^{(0)}$ is independent of $A$. 
\end{theorem}



%
%

From Theorem~\ref{thm:pop_known_pq}, we can calculate lower bounds on the volumes of the basins of attractions of the limit points of the population BCAVI updates. We have the following corollary.
\begin{corollary}\label{cor:volume_pop}
Define the set of initialization points converging to a stationary point $\mathbf{c}$ as
\[
	\mathcal{S}_{\mathbf{c}} := \{v \mid \limsup_{s \rightarrow \infty}n^{-1}\|\psi^{(s)} - \mathbf{c}\|^2 = O(\exp(-\Theta(n\tstar \min \{|\alpha_+|, \alpha_-\}))), \text{ when } \psi^{(0)} = v\}.
\]
Let $\mathfrak{M}$ be some measure on $[0, 1]^n$, absolutely continuous with respect to the Lebesgue measure. Consider the stationary point $\bone$, then
\[
 	\mathfrak{M}(\mathcal{S}_{\bone}) \ge \lim_{\gamma \uparrow 1}\mathfrak{M}(H_+^{\gamma} \cap H_-^{\gamma} \cap [0,1]^n), 
\]
where the half-spaces $H_{\pm}^{\gamma}$ are given as
\[
H_{\pm}^{\gamma} = \big\{x \mid \langle x, \alpha_+ u_1 \pm \alpha_- u_2 \rangle > \frac{n\alpha_+}{2} + \frac{n^{1 - \gamma}}{4t}\big\}.
\]
Similar formulas can be obtained for the other stationary points. 
\end{corollary}

For specific measures $\mathfrak{M}$, one can obtain explicit formulas for these volumes. In practice, these are quite easy to calculate by Monte Carlo simulations. 



In fact, using arguments that go into the proof of Theorem~\ref{thm:pop_known_pq}, we can show that in the large $n$ limit, there are only five stationary points of the mean field log-likelihood, namely $\frac{1}{2}\bone, \bone, \bzero, \icone$, and  $\ictwo$.

We can check that the lower bound required on $n|a^{(0)}_{\pm}|$ by Theorem~\ref{thm:samp_known_pq} always holds when we use initializations of the form $\psi_i^{(0)}\stackrel{iid}{\sim} f_{\mu}$, where $f_{\mu}$ is some distribution with support $[0,1]$ and $\mu\neq\frac{1}{2}$. When $\mu=\frac{1}{2}$, $n|a^{(0)}_{\pm}|=\Theta_P(\sqrt{n}\rho_n)$ which does not satisfy the lower bound. In this case, we have the following theorem showing convergence can happen for a good fraction of the random initializations. 
 
\begin{theorem}[Convergence for random initializations]
	\label{thm:halfcase}
	When $\pstar$ and $\qstar$ are known and $\rho_n \to 0$ at a rate such that $\rho_n\sqrt{n}/\log n \to \infty$, initializing with $\psi_i^{(0)}\sim \text{iid Bernoulli}(\frac{1}{2})$ and using the sample BCAVI updates~\eqref{eq:bcavi_sample_known_pq}, with probability at least $1-\frac{\arctan(c_{\ell})-\arctan(c_{\ell}^{-1})}{\pi}$,
	\begin{align*}
	\| \psi^{(s)} - z_0\|_1 & \leq n\exp(-\rho_n n/\sqrt{c_n}) + \frac{c_n}{n\rho_n} \|\psi^{(s-1)} -z_0 \|_1,
	\end{align*} 
	for $s\geq 3$, $c_n\to \infty$ slowly, where $z_0=\icone$ or $\ictwo$. Here 
	\begin{align*}
	c_{\ell} = \frac{(\pstar-\lambdastar)+c(\lambdastar-\qstar)}{c(\pstar-\lambdastar)+(\lambdastar-\qstar)}, \qquad c=\frac{(\lambdastar-\qstar)(1-\epsilon_n)}{(\pstar-\lambdastar)(1+\epsilon_n)}-\eta
	\end{align*}
	$\epsilon_n\to0$ slowly and $\eta>0$.	
\end{theorem}

\begin{remark}
Note that the convergence probability can also be written as $\frac{1}{2} + 2\arctan(c_{\ell}^{-1})$ which is strictly larger than $1/2$. Furthermore, as $c$ gets closer to 1, $c_{\ell}$ approaches 1 and the convergence probability approaches 1. 
\end{remark}

The next corollary shows that even if we do not know $\pstar$ and $\qstar$ and only have their estimates, the above convergence still holds as long as the estimates are reasonably close to $\pstar$ and $\qstar$.

\begin{corollary}[Using parameter estimates]
	\label{cor:pq_noise}
	The same conclusion as in Theorem~\ref{thm:halfcase} holds if we replace $\pstar, \qstar$ with some $\hat{p}, \hat{q}$ satisfying
	\begin{enumerate}
	\item $\frac{\pstar+\qstar}{2} > \hat{\lambda}$, 
	\item $\hat{\lambda} - \qstar = \Omega(\rho_n) > 0$,
	\end{enumerate}
where $\hat{\lambda}$ is computed using $\hat{p}$ and $\hat{q}$. 
\end{corollary}

\begin{remark}
	\label{rem:pq_noise}
	\begin{enumerate}
	\item
	In practice, $\hat{p}$, $\hat{q}$ can be estimates depending on $A$, then the statements in Corollary~\ref{cor:pq_noise} hold with high probability. 
	\item 
	When $\hat{p}, \hat{q} \asymp \rho_n$, $\hat{p} - \hat{q} = \Omega(\rho_n)>0$, $\hat{\lambda}$ lies between $(\hat{p}+\hat{q})/2$ and $\hat{q}$ as suggested by Proposition~\ref{prop:lambda_lbub}. The conditions in Corollary~\ref{cor:pq_noise} imply an upper bound on $\hat{p}$ and a lower bound on $\hat{q}$. Similar constraints hold if $\hat{q} - \hat{p} = \Omega(\rho_n)>0$. An example of the estimate regime is shown in Figure~\ref{fig:pqnoise_corollary}, where $\pstar=0.3$, $\qstar=0.1$, and the yellow area contains $\hat{p}$, $\hat{q}$ such that $\frac{\pstar+\qstar}{2} > \hat{\lambda} > \qstar$.
	\end{enumerate}
\end{remark}

\begin{figure}[ht]
	\centering
	\includegraphics[width=0.4\textwidth]{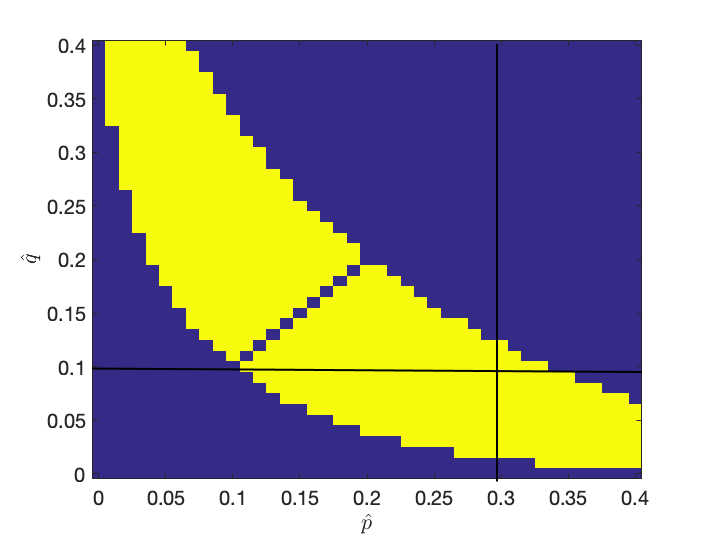} 
	\caption{For $\pstar=0.3$, $\qstar=0.1$, the yellow area shows where $\frac{\pstar+\qstar}{2} > \hat{\lambda} > \qstar$ is satisfied.}
	\label{fig:pqnoise_corollary}
\end{figure}




\subsection{Unknown \texorpdfstring{$\pstar, \qstar$}{}:}
\label{subsec:unknownpq}
In this case, the model parameters $p$ and $q$ are updated jointly with $\psi$. The full BCAVI updates are
\begin{align}\label{eq:bcavi_unknown_pq}
p^{(s)} &= \frac{(\psi^{(s-1)})^\top A \psi^{(s-1)} + (\bone - \psi^{(s-1)})^\top A (\bone - \psi^{(s-1)})}{(\psi^{(s-1)})^\top (J - I) \psi^{(s-1)} + (\bone - \psi^{(s-1)})^\top (J - I) (\bone - \psi^{(s-1)})}, \\ \nonumber
q^{(s)} &= \frac{(\psi^{(s-1)})^\top A (\bone - \psi^{(s-1)})}{(\psi^{(s-1)})^\top (J - I) (\bone - \psi^{(s-1)})}, \\ \nonumber
t^{(s)} &= \frac{1}{2} \log \bigg(\frac{p^{(s)}(1 - q^{(s)})}{q^{(s)}(1 - p^{(s)})}\bigg), \,\, \lambda^{(s)} = \frac{1}{2t^{(s)}} \log \bigg(\frac{1 - q^{(s)}}{1 - p^{(s)}}\bigg), \\ \nonumber
\xi^{(s)} &= 4t^{(s)} (A - \lambda^{(s )}(J - I))(\psi^{(s - 1)} - \frac{1}{2}\bone).
\end{align}
Similar to before, $\pstar\asymp \qstar \asymp \rho_n$ with $\rho_n$ possibly going to 0. In the population version, we would replace $A$ with $\E(A \mid Z) = P - p I.$

In this case with unknown $\pstar$, $\qstar$, our next result shows that $\frac{1}{2}\bone$ changes from a saddle point (Proposition~\ref{prop:saddle}) to a local maximum. 
\begin{proposition}\label{prop:local_max}
	Let $n \ge 2$. Then $(\psi, p, q) = (\frac{1}{2}\bone, \frac{\bone^\top A \bone}{n(n - 1)}, \frac{\bone^\top A \bone}{n(n - 1)})$ is a strict local maximum of the mean field log-likelihood.
\end{proposition}

Since $\pstar$, $\qstar$ and $\psi$ are unknown and need to be estimated iteratively, we have the following updates for $\phat$ and $\qhat$ given the initialization $\psi^{(0)}$ and show that they can be written in terms of the projection of the initialization in the principal eigenspace of $P$. 

\begin{lemma}
	\label{lem:unknownpupdate}
Let $x=(\psi^{(0)})^T\psi^{(0)}+(\bone-\psi^{(0)})^T(\bone-\psi^{(0)})$ and $y=2(\psi^{(0)})^T(\bone-\psi^{(0)})=n-x$. Projecting $\psi^{(0)}$ onto $u_1$ and $u_2$ and writing $\psi^{(0)}=\zeta_1u_1+\zeta_2u_2+w$, where $w\in \spn\{u_1,u_2\}^\perp$, then
	\begin{align}
	\phat&=\frac{\pstar+\qstar}{2}+\frac{(\pstar-\qstar)(\zeta_2^2-x/2n^2)}{\zeta_1^2+(1-\zeta_1)^2-x/n^2} + O_P(\sqrt{\rho_n}/n),	\nonumber\\
	\qhat&=\frac{\pstar+\qstar}{2}-\frac{(\pstar-\qstar)(\zeta_2^2+y/2n^2)}{2\zeta_1(1-\zeta_1)-y/n^2} + O_P(\sqrt{\rho_n}/n).
	\label{eq:pqupdate}
	\end{align}
\end{lemma}

	 Since $(\psi^{(0)})^T(\bone-\psi^{(0)})>0$, we have $\zeta_1(1-\zeta_1)\geq \zeta_2^2$.
This gives:
\begin{align}\label{eq:pqbound}
\phat\in \left(\frac{\pstar+\qstar}{2}+O_P(\sqrt{\rho_n}/n),\pstar\right],\qquad
\qhat\in \left[\qstar,\frac{\pstar+\qstar}{2}+O_P(\sqrt{\rho_n}/n)\right).
\end{align}

It is interesting to note that $\phat$ is always smaller than $\qhat$ except when it is $O(\sqrt{\rho_n}/n)$ close to $(\pstar+\qstar)/2$. In that regime, one needs to worry about the sign of $t$ and $\lambda$. In all other regimes, $t,\lambda$ are positive.

Using the update forms in Lemma~\ref{lem:unknownpupdate}, the following result shows that the stationary points of the population mean field log-likelihood lie in the principle eigenspace $\spn\{u_1,u_2\}$ of $P$ in a limiting sense.

\begin{proposition}\label{prop:stationary_characterization}
	Consider the case with unknown $\pstar$, $\qstar$ and $\rho_n\to0$, $n\rho_n\to\infty$. Let $(\psi, \tilde{p}, \tilde{q})$ be a stationary point of the population mean field log-likelihood. If $\psi=\psi_u+\psi_{u^\perp}$, where $\psi_u\in \text{span}\{u_1,u_2\}$ and $\psi_{u^\perp} \perp\text{span}\{u_1,u_2\}$, then $\|\psi_{u^\perp}\|=o(\sqrt{n})$ as $n\to \infty$.
\end{proposition}

Lemma~\ref{lem:unknownpupdate} basically shows that if $\zeta_2$ is vanishing, then $\phat$ and $\qhat$ concentrates around the average of the conditional expectation matrix, i.e. $(\pstar+\qstar)/2$. The next result shows that if one uses independent and identically distributed initialization, then $\zeta_2$ is indeed vanishing. This is not surprising, since $\zeta_2$ measures correlation with the second eigenvector of $P$ $u_2$ which is basically the $\icone - \ictwo$ vector. 

Consider a simple random initialization, where the entries of $\psi^{(0)}$ are i.i.d with mean $\mu$, we show that the update converges to $\frac{1}{2}\bone$ with small deviations within one update. This shows the futility of random initialization.

\begin{lemma}\label{lem:iidpsi-unknown}
	Consider the initial distribution $\psi^{(0)}_i\stackrel{iid}{\sim} f_\mu$ where $f$ is a distribution supported on $(0,1)$ with mean $\mu$. If $\mu$ is bounded away from $0$ and $1$ and $n \rho_n = \Omega(\log^2 n)$, using the updates in~\eqref{eq:bcavi_unknown_pq}, then $\|\psi^{(1)}-\frac{1}{2}\bone\|_2 = O_P(1)$, $\|\psi^{(s)}-\frac{1}{2}\bone\|_2 = O_P(\sqrt{\rho_n})$ for all $s\geq 2$. 
\end{lemma}	


Perhaps, it is also instructive to analyze the case where the initialization is in fact correlated with the truth, i.e. $E[\psi^{(0)}_i]=\mu_{\sigma_i}$. To this end, we will consider the following initialization scheme.
\begin{lemma}\label{lem:goodinitunknown}
Consider an initial $\psi^{(0)}$ such that 
\begin{align}\label{eq:zeta-eqmean}
\zeta_1&=\frac{(\psi^{(0)})^T  \bone}{n}=\frac{\mu_1+\mu_2}{2}+O_P(1/\sqrt{n}),\nonumber\\
\zeta_2&=\frac{(\psi^{(0)})^T u_2}{n}=\frac{\mu_1-\mu_2}{2}+O_P(1/\sqrt{n}).
\end{align} 

If $\mu_1, \mu_2$ are bounded away from 0 and 1 and satisfy 
\begin{align}\label{eq:museparation}
|\mu_1-\mu_2| >\max\left(2|\mu_1+\mu_2-1|+O_P\left(\frac{\sqrt{\rho_n\log^2 n/n}}{\pstar-\qstar}\right),\left(\frac{\rho_n\log n}{n(\pstar-\qstar)^2}\right)^{1/3}\right) ,
\end{align} \bk
and $n\rho_n =\Omega(\log^2 n)$, then $\psi^{(1)} = \icone + O_P(\exp(-\Omega(\log n)))$ or $ \ictwo + O_P(\exp(-\Omega(\log n)))$, where the error term is uniform for all the coordinates.
\end{lemma}
\begin{remark}
	
\begin{enumerate}
	\item The lemma states that provided the separation between $\pstar$ and $\qstar$ does not vanish too fast, if the initial $\psi^{(0)}$ is centered around two slightly different means, e.g., $\mu_1=1/2+\epsilon_n$ and $\mu_2=1/2-\epsilon_n$ for some constant $\epsilon_n\to 0$, then we converge to the truth within one iteration. 
	
	\item 
	
	Now we have $\|\psi^{(1)} - z_0\|_1 = o_P(n)$, which satisfies the condition in~\cite{zhang2017theoretical} ($\|\psi^{(1)} - z_0\|_1 \leq c_0 n$ for some constant $c_0$ small enough with high probability). The rest of their regularity conditions can also be checked. Thus for $s> 1$,
	\begin{align*}
	\|\psi^{(s)} - z_0\|_1 \leq n\exp(-C_1n \rho_n) + \frac{C_2}{\sqrt{n\rho_n}} \|\psi^{(s-1)} - z_0\|_1 
	\end{align*} 
	with high probability. 

\end{enumerate}
\end{remark}

\section{Numerical results}\label{sec:simu}
In Figure~\ref{fig:region}-(a), we have generated a network from an SBM with parameters $\pstar = 0.4, \qstar = 0.025$, and two equal sized blocks of $100$ nodes each. We generate $5000$ initializations $\psi^{(0)}$ from $\mathrm{Beta}(\alpha,\beta)^{\otimes n}$ (for four sets of $\alpha$ and $\beta$) and map them to  $a_{\pm 1}^{(0)}$. We perform sample BCAVI updates on $\psi^{(0)}$ with known $\pstar, \qstar$ and color the points in the $a_{\pm 1}^{(0)}$ co-ordinates according the limit points they have converged to. In this case, $\alpha_+ > 0$, hence based on Theorems~\ref{thm:pop_known_pq} and~\ref{thm:samp_known_pq}, we expect points with $a_{+1}^{(0)}a_{-1}^{(0)} < 0$ to converge to the ground truth (colored green or magenta) and those with $a_{+1}^{(0)}a_{-1}^{(0)} > 0$ to converge to $\bzero	$ or $\bone$. As expected, points falling in the center of the first and third quadrants have converged to $\bzero$ or $\bone$. The points converging to the ground truth lie more toward the boundaries but mostly remain in the same quadrants, suggesting possible perturbations arising from the sample noise and small network size. We see that this issue is alleviated when we increase $n$. 


\begin{figure}[ht]
	\centering
	\begin{tabular}{cc}
	\includegraphics[width=0.51\textwidth]{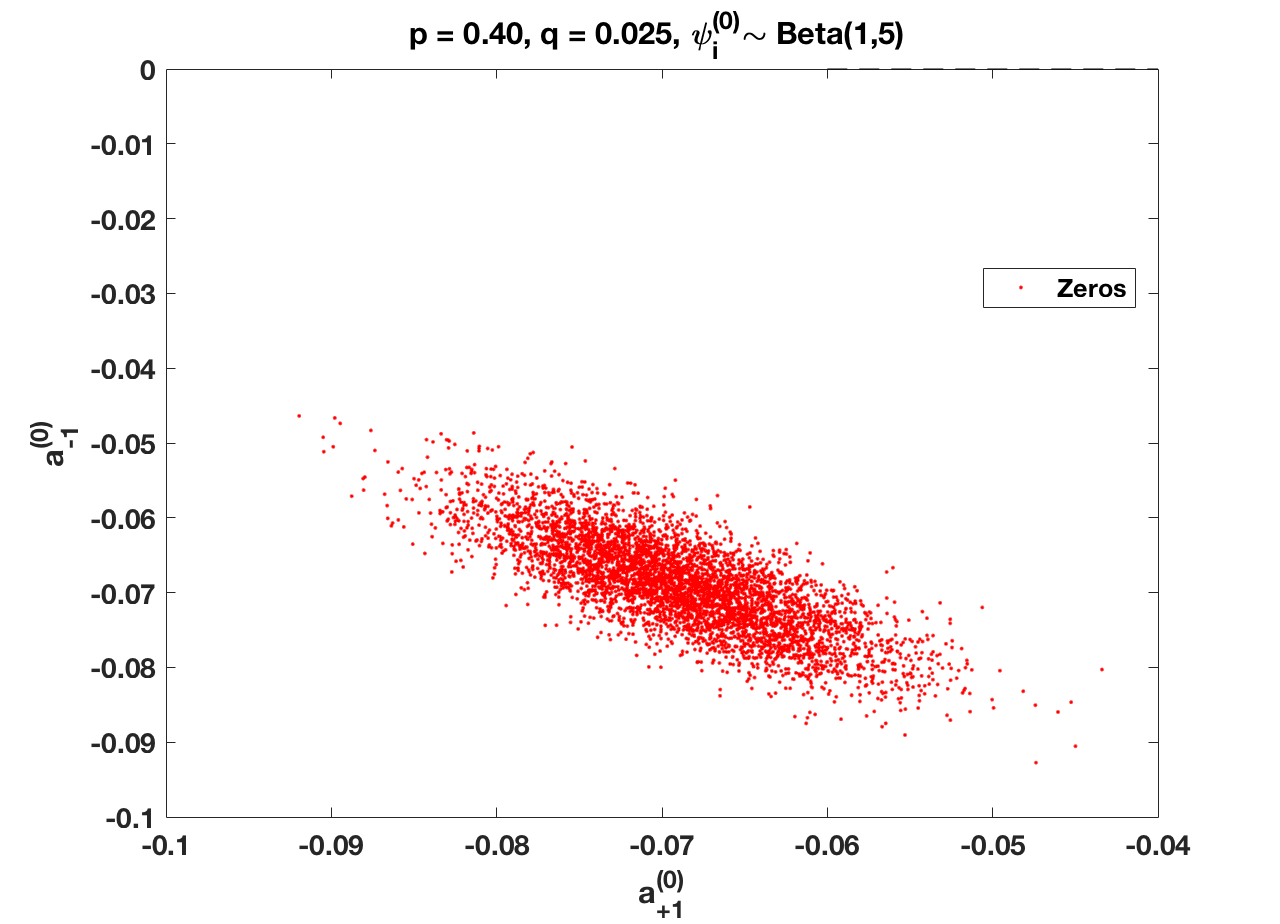} & \includegraphics[width=0.5\textwidth]{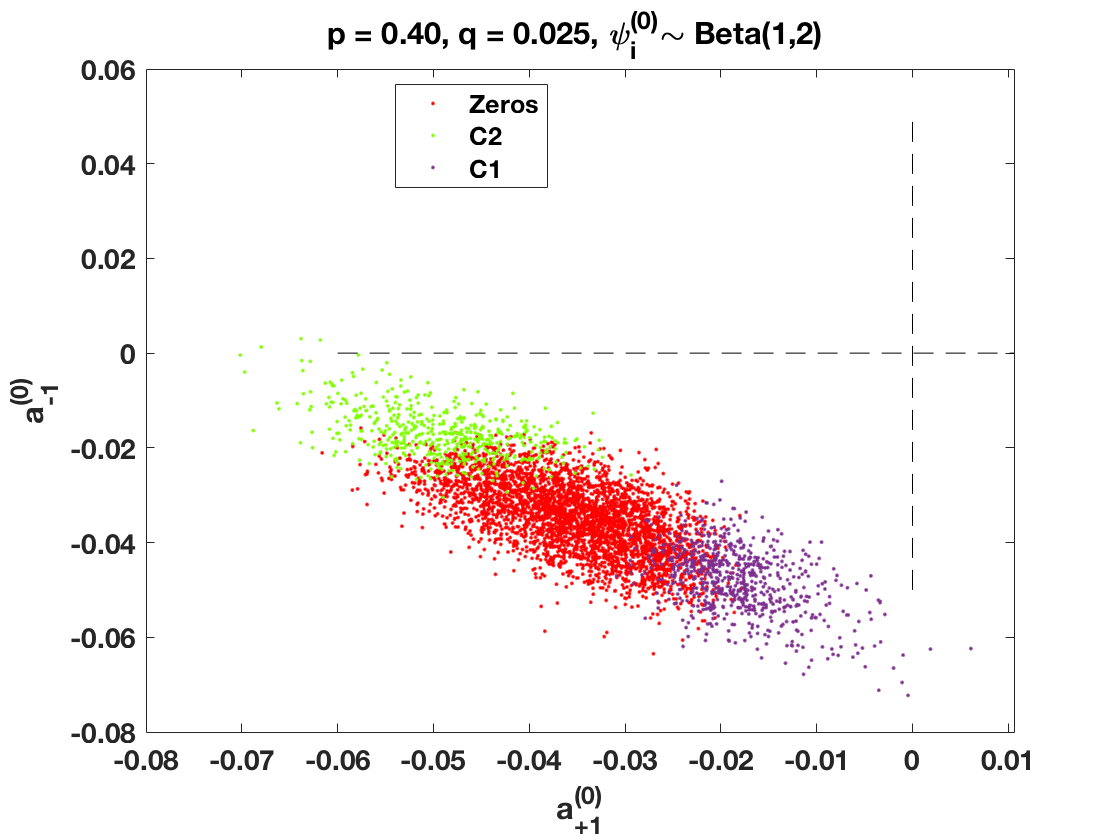}\\
	(a)&(b)\\
	\includegraphics[width=0.49\textwidth]{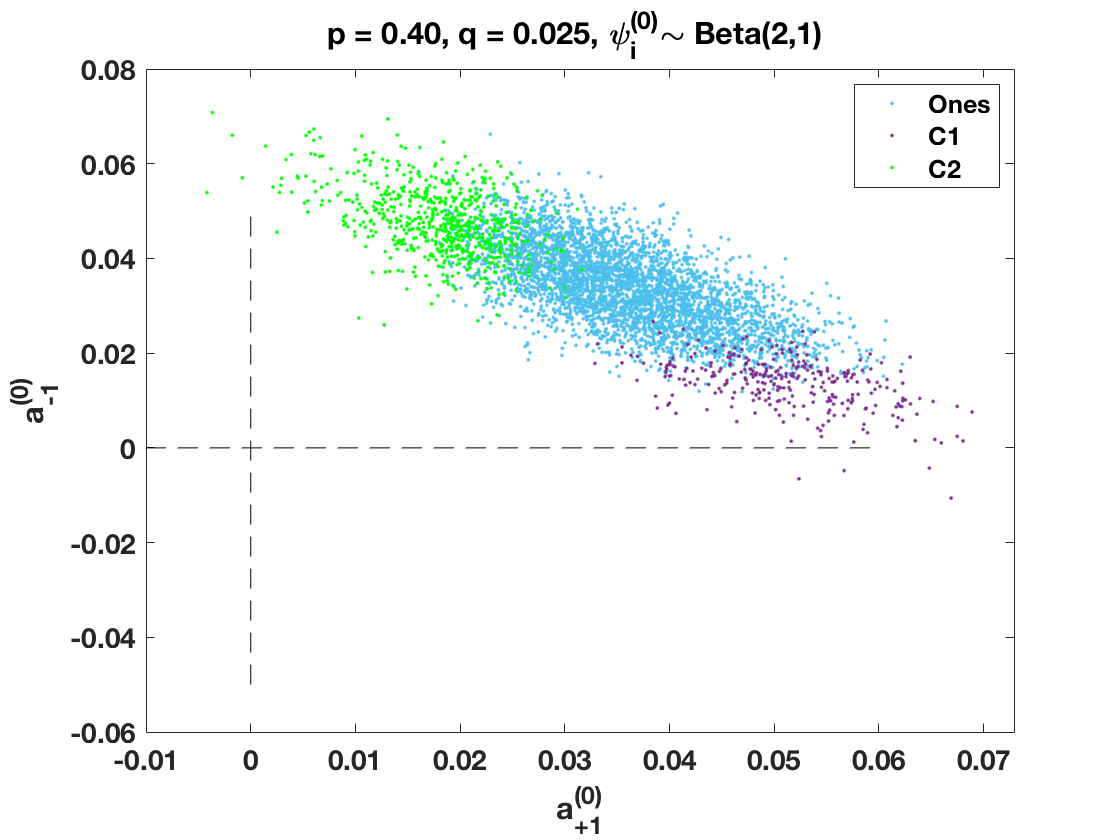} & \includegraphics[width=0.5\textwidth]{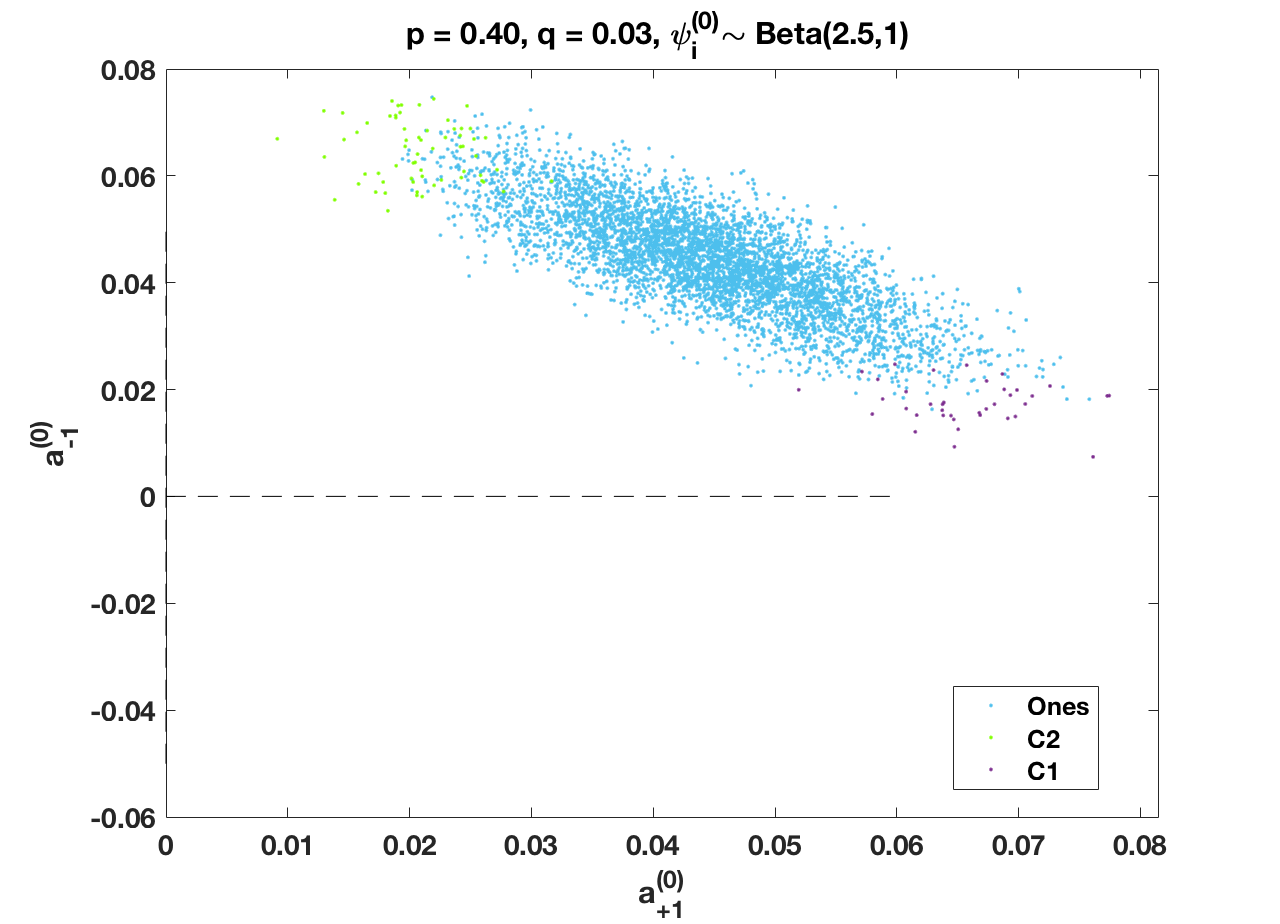}\\
	(c)& (d)
	\end{tabular}
	\caption{ $n = 200$ and $5000$,  $\psi^{(0)}\sim\mathrm{Beta}(\alpha,\beta)^{\otimes n}$ for various values of $\alpha$ and $\beta$. These $\psi^{(0)}$ are mapped to $(a_{+1}^{(0)}, a_{-1}^{(0)})$ (see ~\eqref{eq:main_decomp}) and plotted.  $C_1$ (magenta) and $C_2$ (green) correspond to the limit points $\icone$ and $\ictwo$. Other limit points are `Ones', i.e. $\bone$ (blue) and `Zeros', i.e. $\bzero$ (red). }
	\label{fig:region}
\end{figure}
The notable thing is, in Figure~\ref{fig:region}-(a) and (d), the Beta distribution has mean $0.16$ and $0.71$ respectively. So the initialization is more skewed towards values that are closer to zero or closer to one. In these cases most of the random runs converge to the all zeros or all ones, with very few converging to the ground truth. However, for Figure~\ref{fig:region}-(b) and (d), the mean of the Beta is $0.3$ and $0.7$, and we see considerably more convergences to the ground truth. Also, (b) and (d) are, in some sense, mirror images of each other, i.e. in one, the majority converges to $\bzero$; whereas in the other, the majority converges to $\bone$.

In Figure~\ref{fig:pqnoise}, we examine whether convergence can hold even when the exact values of $\pstar$, $\qstar$ are unknown using the initiliazation scheme in Theorem~\ref{thm:halfcase} and Corollary~\ref{cor:pq_noise}. In each heatmap, the dashed lines indicate the true parameter values used to generate an adjacency matrix $A$. The heatmap contains pairs of $\hat{p}, \hat{q}$ that we use in the sample BCAVI updates~\eqref{eq:bcavi_sample_known_pq} for fixed parameters initialized with $\psi_i^{(0)}\sim\text{iid Bernoulli}(\frac{1}{2})$. For each pair of parameters, we use 50 such random initializations and compute the average clustering accuracy. In both cases, we can see that as long as the parameter estimates fall into a reasonable range around the true values, convergence to the ground truth happens for a high fraction of the random initializations. The plots are symmetric in terms of $\hat{p}$ and $\hat{q}$, suggesting the estimates do not have to respect the relationship $\hat{p}>\hat{q}$ as discussed in Remark~\ref{rem:pq_noise}.  

\begin{figure}[ht]
	\centering
	\begin{tabular}{cc}
		\includegraphics[width=0.5\textwidth]{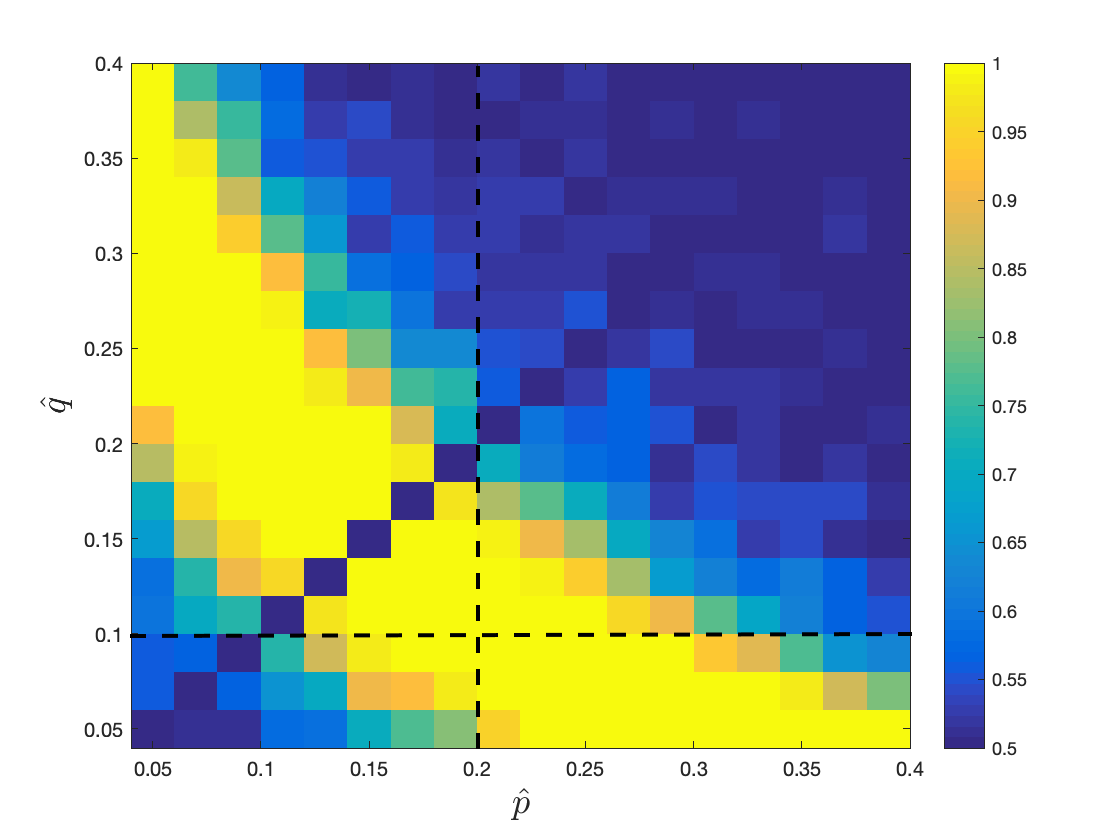} & \includegraphics[width=0.5\textwidth]{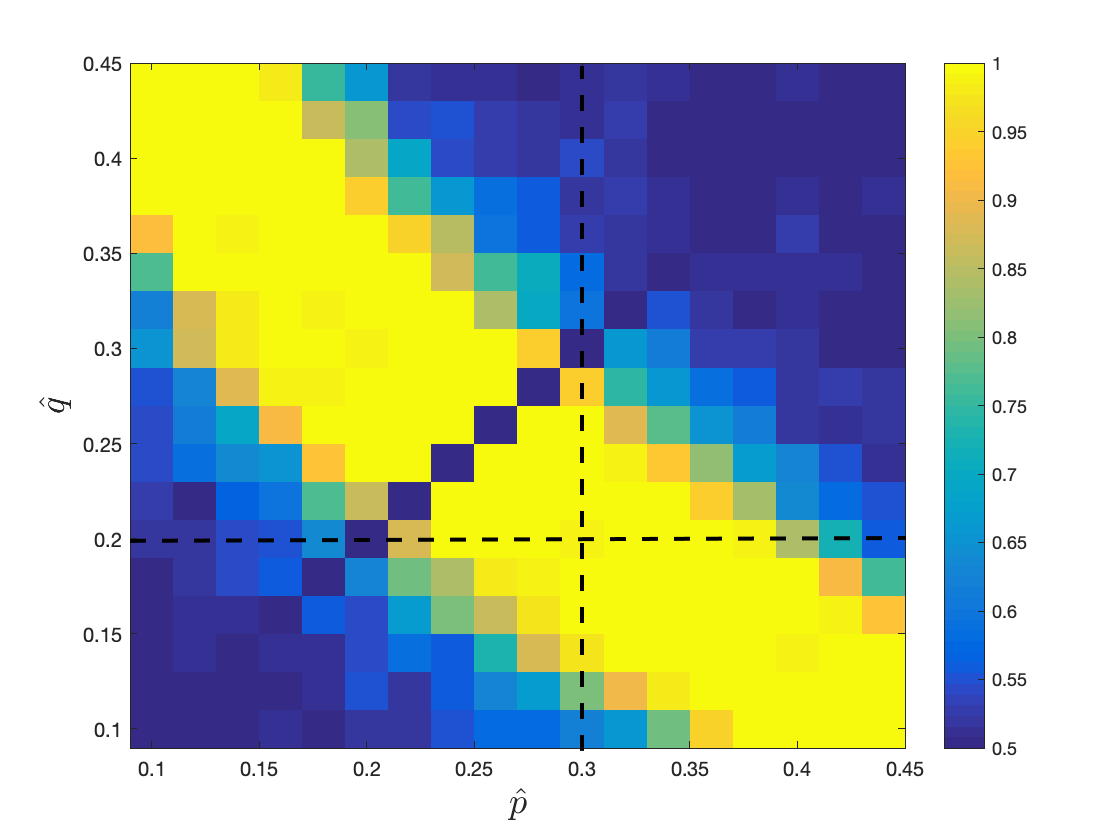}\\
		(a)&(b)\\
	\end{tabular}
	\caption{Average clustering accuracy using 50 random initializations $\psi_i^{(0)}\sim\text{iid Bernoulli}(\frac{1}{2})$ and different $\hat{p}, \hat{q}$ values in the BCAVI updates with fixed parameters. The dashed lines show the true parameter values, (a) $\pstar=0.2$, $\qstar=0.1$, (b) $\pstar=0.3$, $\qstar=0.2$. } 
\label{fig:pqnoise}
\end{figure}

In Figure~\ref{fig:error}, we examine initializations of the type described in Lemma~\ref{lem:goodinitunknown} and the resulting estimation error. For each $c_0$, we initialize $\psi^{(0)}$ such that $\E(\psi^{(0)})= (1/2+c_0)\icone + (1/2-c_0)\ictwo$ with iid noise. The y-axis shows the average distance between $\psi^{(20)}$ and the true $z_0$ from 500 such initializations, as measured by $\Vert \psi^{(20)} - z_0\Vert_1/n$.  For every choice of $\pstar, \qstar$, a network of size 400 with two equal sized blocks was generated. In all cases, sufficiently large $c_0$ guarantees convergence to the truth. We also observe that the performance deteriorates when $\pstar-\qstar$ becomes small, either when $\pstar$ decreases or when the network becomes sparser.

\begin{figure}[h]
	\centering
	\includegraphics[width=0.5\textwidth]{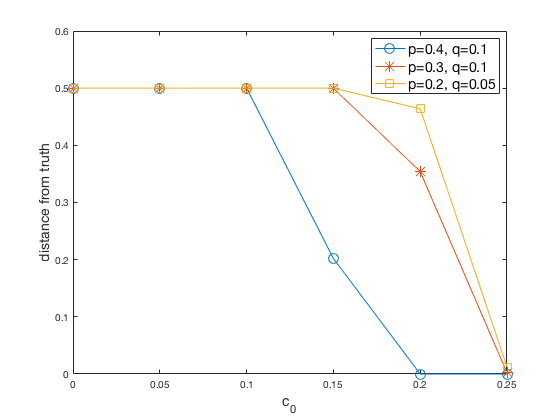}
	\caption{Average distance between the estimated $\psi$ and the true $z_0$ with respect to $c_0$, where $\E(\psi^{(0)})= (1/2+c_0)\icone + (1/2-c_0)\ictwo$.}
	\label{fig:error}
\end{figure}

\section{Discussion}\label{sec:conc}
In this paper, we work with the BCAVI mean field variational algorithm for a simple two class stochastic blockmodel with equal sized classes. Mean field methods are used widely for their scalability. However, existing theoretical works typically analyze the behavior of the global optima, or the local convergence behavior when initialized near the ground truth. In the simple setting considered, we show two interesting results. First, we show that, when the model parameters are known, random initializations centered around half converge to the ground truth a good fraction of time. The same convergence holds if some reasonable estimates of the model parameters are known and held fixed throughout the updates.  In contrast, when the parameters are not known and estimated iteratively with the mean field parameters, we show that a random initialization converges, with high probability, to a meaningless local optimum. This shows the futility of using multiple random initializations when no prior knowledge is available. 

In view of recent works on the optimization landscape for Gaussian mixtures \citep{jin2016local,xu2016global}, we would like to comment that, despite falling into the category of latent variable models, the SBM has fundamental differences from Gaussian mixtures which require different analysis techniques. The posterior probabilities of the latent labels in the latter model can be easily estimated when the parameters are known, whereas this is not the case for SBM since the posterior probability $\P(Z_i|A)$ depends on the entire network. The significance of the results in Section~\ref{subsec:knownpq} lies in characterizing the convergence of label estimates given the correct parameters for general initializations, which is different from the type of parameter convergence shown in~\citep{jin2016local,xu2016global}. Furthermore, as most of the existing literature for the SBM focuses on estimating the labels first, our results provide an important complementary direction by suggesting that one could start with parameter estimation instead. 

While we only show results for two classes, we expect that our main theoretical results generalize well to $K>2$ and will leave the analysis for future work. As an illustration, consider a setting similar to that of Figure~\ref{fig:region} but for $n=450$ with $K=3$ equal sized classes. $\pstar=0.5$, $\qstar=0.01$ are known and $\psi^{(0)}$ is initialized with a Dirichlet$(0.1, 0.1, 0.1)$ distribution. Each row of the matrix in Figure~\ref{fig:sep_K3} represents a stationary cluster membership vector from a random initialization. 

In Figure~\ref{fig:sep_K3}, all 1000 random initializations converge to stationary points $\psi$ lying in the span of $\{\icone, \ictwo, \icthree\}$, which are the membership vectors for each class. We represent the node memberships with different colors, and there are $1 + \binom{3}{2} = 4$ different types of stationary points, not counting label permutations. Another stationary point (the all ones vector that puts everyone in the same class) can be obtained with other initialization schemes, e.g., when the rows of $\psi^{(0)}$ are identical. For a general $K$- blockmodel, we conjecture that the number of stationary points grows exponentially with $K$. Similar to Figure~\ref{fig:region}, a significant fraction of the random initializations converge to the ground truth. On the other hand, when $\pstar,\qstar$ are unknown, random initializations always converge to the uninformative stationary point $(1/3, 1/3, 1/3)$, analogous to Lemma~\ref{lem:iidpsi-unknown}.
\begin{figure}[h]
	\centering
	\includegraphics[width=0.5\textwidth]{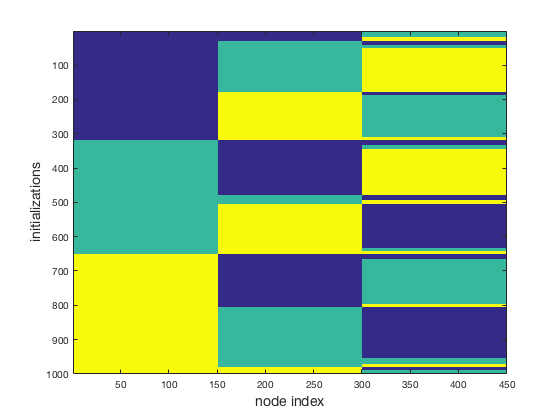}  \\
\caption{Convergence to stationary points for known $\pstar,\qstar$, $K=3$. Rows permuted for clarity.}
\label{fig:sep_K3}
\end{figure}


\section*{Acknowledgement}
SSM thanks Professor Peter J. Bickel for helpful discussions. PS is partially funded by NSF grant DMS1713082. YXRW is supported by the ARC DECRA Fellowship.


\appendix
\section*{Appendix A.}
\label{app:theorem}



This appendix provides derivation of stationarity equations for the mean field log-likelihood and the proofs of our main results.
\section{The Variational principle and mean field}
We start with the following simple observation:
\begin{align*}
\log P(A; B, \pi) &= \log \sum_Z P(A,Z; B,\pi)= \log \left(\sum_Z \frac{P(A,Z; B,\pi)}{\psi(Z)}\psi(Z) \right)\\
&\overset{\text{(Jensen)}}{\ge} \sum_Z \log \left(\frac{P(A,Z; B,\pi)}{\psi(Z)} \right)\psi(Z) \hskip20pt\forall \psi \text{ prob. on }\mathcal{Z}.
\end{align*}
In fact, equality holds for $\psi^*(Z) = P(Z|A; B,\pi)$. Therefore, if $\Psi$ denotes the set of all probability measures on $\mathcal{Z}$, then 
\begin{equation}
\log P(A; B,\pi) = \max_{\psi \in \Psi} \sum_Z \log \left(\frac{P(A,Z; B,\pi)}{\psi(Z)} \right)\psi(Z).
\end{equation}
The crucial idea from variational inference  is to replace the set $\Psi$ above by some easy-to-deal-with subclass $\Psi_0$ to get a lower bound on the log-likelihood.
\begin{equation}
\label{eq:mf_lowerbound}
\log P(A; B,\pi) \ge \max_{\psi \in \Psi_0 \subset \Psi} \sum_Z \log \left(\frac{P(A,Z; B,\pi)}{\psi(Z)} \right)\psi(Z).
\end{equation}
Also the optimal $\psi_\star\in \Psi_0$ is a potential candidate for an estimate of $P(Z|A; B,\pi)$. Estimating $P(Z|A; B,\pi)$ is profitable since then we can obtain an estimate of the community membership matrix by setting $Z_{ia} = 1$ for the $i$th agent where 
\begin{equation}
a = \arg \max_b P(Z_{ib} = 1 | A; B, \pi).
\end{equation}
The goal now has become optimizing the lower bound in~\eqref{eq:mf_lowerbound}.
\section{Derivation of stationarity equations}\label{sec:deriv_stat_points}
\begin{align}\label{eq:first_derivative}\nonumber
 & \frac{\partial \ell}{\partial \psi_i} = 4t \sum_{j : j \ne i} (\psi_j - \frac{1}{2})(A_{ij} - \lambda) - \log\bigg( \frac{\psi_i}{1 - \psi_i} \bigg),\\ \nonumber
 & \frac{\partial \ell}{\partial p} = \frac{1}{2} \sum_{i, j : i \ne j} (\psi_i \psi_j + (1 - \psi_i)(1 - \psi_j)) \bigg(A_{ij} \bigg(\frac{1}{p} + \frac{1}{1 - p}\bigg) - \frac{1}{1 - p}\bigg),  \\
 & \frac{\partial \ell}{\partial q} = \frac{1}{2} \sum_{i, j : i \ne j} (\psi_i (1 - \psi_j) + (1 - \psi_i)\psi_j) \bigg(A_{ij} \bigg(\frac{1}{q} + \frac{1}{1 - q}\bigg) - \frac{1}{1 - q}\bigg).
\end{align}

Therefore
\begin{align}\label{eq:second_derivative}\nonumber
\frac{\partial^2 \ell}{\partial \psi_j \partial \psi_i} &= 4t (A_{ij} - \lambda) (1 - \delta_{ij}) - \frac{1}{\psi_i(1 - \psi_i)} \delta_{ij}, \\ \nonumber
\frac{\partial^2 \ell}{\partial \psi_i \partial p} &= \frac{1}{2} \sum_{j : j \ne i} \bigg(\frac{1}{2} - \psi_j \bigg) \bigg(A_{ij} \bigg(\frac{1}{p} + \frac{1}{1 - p}\bigg) - \frac{1}{1 - p}\bigg),  \\ \nonumber
\frac{\partial^2 \ell}{\partial \psi_i\partial q} &= \frac{1}{2} \sum_{j : j \ne i} \bigg(\psi_i - \frac{1}{2}\bigg) \bigg(A_{ij} \bigg(\frac{1}{q} + \frac{1}{1 - q}\bigg) - \frac{1}{1 - q}\bigg),\\ \nonumber
\frac{\partial^2 \ell}{\partial p^2} &= \frac{1}{2} \sum_{i, j : i \ne j} (\psi_i \psi_j + (1 - \psi_i)(1 - \psi_j)) \bigg(A_{ij} \bigg(-\frac{1}{p^2} + \frac{1}{(1 - p)^2}\bigg) - \frac{1}{(1 - p)^2}\bigg), \\ \nonumber
\frac{\partial^2 \ell}{\partial q^2} &= \frac{1}{2} \sum_{i, j : i \ne j} (\psi_i (1 - \psi_j) + (1 - \psi_i)\psi_j) \bigg(A_{ij} \bigg(-\frac{1}{q^2} + \frac{1}{(1 - q)^2}\bigg) - \frac{1}{(1 - q)^2}\bigg),\\
\frac{\partial^2 \ell}{\partial q\partial p} &= 0.
\end{align} 

\section{Proofs of main results}
\begin{proof}[Proof of Proposition~\ref{prop:bd_t_lambda}]
	For any $a > b > 0$, we have
	\[
	\frac{a - b}{a} < \log\bigg(\frac{a}{b}\bigg) < \frac{a - b}{b},
	\]
	which can be proved using the inequality $\log(1 + x) < x$ for $x > -1, x \ne 0$. Therefore
	\[
	\frac{p - q}{p} < \log\bigg(\frac{p}{q}\bigg) < \frac{p - q}{q}, \,\,\,\text{and }\,\,\, \frac{p - q}{1 - q} < \log\bigg(\frac{1 - q}{1 - p}\bigg) < \frac{p - q}{1 - p}.
	\]
	So
	\[
	\frac{(p - q) (1 + p - q)}{2(1 - q) p} < t = \frac{1}{2} \bigg(\log\bigg(\frac{p}{q}\bigg) + \log\bigg(\frac{1 - q}{1 - p}\bigg) \bigg) < \frac{(p - q) (1 - p + q)}{2(1 - p) q},
	\]
	and
	\[
	q = \frac{\frac{p - q}{1 - q}}{\frac{p - q}{q} + \frac{p - q}{1 - q}} < \lambda = \frac{\log(\frac{1 - q}{1 - p})}{\log(\frac{p}{q}) + \log(\frac{1 - q}{1 - p})} <  \frac{\frac{p - q}{1 - p}}{\frac{p - q}{p} + \frac{p - q}{1 - p}} = p.
	\]

\end{proof}

\begin{proof}[Proof of Proposition~\ref{prop:lambda_lbub}]
	Let $y=(p-q)/(1-p) >0$. 
	We will use the well known inequalities~\citep{vu17162}:
	\begin{align}
	\log(1+y)&\geq \frac{2y}{2+y}\geq \frac{y}{1+y}, \label{eq:logub}\\
	\log(1+y)&\leq y-\frac{y^2}{2(1+y)}\label{eq:loglb}
	\end{align}
	Using Eq~\eqref{eq:loglb},
	\begin{align*}
	\lambda&=\frac{\log\frac{1-q}{1-p}}{\log\frac{p}{q}+\log\frac{1-q}{1-p}}\geq \frac{y}{(1+y)\log\frac{p}{q}+y}\geq\frac{(p-q)}{\log\frac{p}{q}+(p-q)}
	\end{align*}
	Using Eq~\eqref{eq:logub} we get:
	\begin{align*}
	\lambda-q&\geq \frac{(p-q)-q\log(p/q)-O(\rho_n^2)}{\log\frac{p}{q}+(p-q)}\\
	&\geq  \frac{(p-q)-q\left(\frac{p-q}{q}-\frac{(p-q)^2}{2pq}\right)-O(\rho_n^2)}{\log\frac{p}{q}+(p-q)}\\
	&\geq  \frac{\frac{(p-q)^2}{2p}-O(\rho_n^2)}{\log\frac{p}{q}+O(\rho_n)}=\Omega(\rho_n)\\
	\end{align*}
	The last step is true since $p-q=\Omega(\rho_n)$.
	
	Now we prove Eq~\eqref{eq:lambdaub_sep}. Let $x:=p/q-1=\Omega(1)$, since $p-q=\Omega(\rho_n)$.
	\begin{align}
	\lambda&\leq \frac{p-q}{(1-p)\log\frac{p}{q}+(p-q)}\nonumber\\
	\frac{p+q}{2}-\lambda &\geq\frac{\frac{p+q}{2}\log(p/q)-(p-q)-O(\rho_n^2)}{(1-p)\log\frac{p}{q}+(p-q)}\nonumber\\
	&=q\frac{(1+x/2)\log (1+x)-x-O(\rho_n)}{\log(p/q)+O(\rho_n)}\label{eq:lambdalb2}
	\end{align}
	Consider the function $h(x)$ defined below, where $x=p/q-1=\Omega(1)$.
	\begin{align*}
	h(x)&=(2+x)\log(1+x)-2x\\
	h'(x)&=\log(1+x)+\frac{2+x}{1+x}-2=\log(1+x)-\frac{x}{1+x}\\
	&\geq \frac{2x}{2+x}-\frac{x}{1+x}=\frac{x^2}{(2+x)(1+x)}=\Omega(1)
	\end{align*}
	Plugging into Eq~\eqref{eq:lambdalb2} we get:
	\begin{align*}
	\frac{p+q}{2}-\lambda&\geq q\frac{h(x)-O(\rho_n)}{2\log(p/q)+O(\rho_n)}=\Omega(\rho_n)
	\end{align*}
\end{proof}

\subsection{Proofs of results in Section~\ref{subsec:knownpq}}

\begin{proof}[Proof of Proposition~\ref{prop:saddle}]
	That $\psi = \frac{1}{2}\bone$ is a stationary point is obvious from the stationarity equations \eqref{eq:first_derivative}. The eigenvalues of $-4I + 4\tstar M$, the Hessian at $\frac{1}{2}\bone$, are $h_i = - 4 + 4\tstar \nu_i$. We have $\nu_1 = n\alpha_+ - (\pstar - \lambdastar) = \Theta(n)$, and hence so is $h_1$. Also, $ \pstar - \lambdastar > 0$, so that $\nu_3 < 0$, and hence $h_3 < 0$. Thus we have two eigenvalues of the opposite sign.
\end{proof}

\begin{proof}[Proof of Theorem~\ref{thm:pop_known_pq}]
	From \eqref{eq:main_decomp}, we have
	\[
	\psi^{(s+1)}_i = g(n a_{\sigma_i}^{(s)} + b_i^{(s)}) = g(na_{\sigma_i}^{(s)}) + \delta_i^{(s)},
	\]
	where $|\delta_i^{(s)}| = O(\exp(-n |a_{\sigma_i}^{(s)}|))$, where we have used the fact that
	\[
	g(nx + y) - g(nx) = g(nx)g(nx+y)(e^y - 1)\exp(-(nx+y)).
	\]
	Writing as a vector, we have
	\begin{equation}\label{eq:psi_approx}
	\psi^{(s+1)} = g(n a_{+1}^{(s)}) \bone_{\mathcal{C}_1} + g(n a_{-1}^{(s)}) \bone_{\mathcal{C}_2} + \delta^{(s)},
	\end{equation}
	where $\|\delta^{(s)}\|_{\infty} = \max_i |\delta_i^{(s)}| = O(\exp(-n \min\{| a_{+1}^{(s)}|,  |a_{-1}^{(s)}|\}))$. Note that by our assumption, $\|\delta^{(0)}\|_\infty = O(\exp(-n \min\{| a_{+1}^{(s)}|,  |a_{-1}^{(s)}|\})) = o(1)$. Now
	\[
	\zeta^{(s+1)}_1 = \frac{\langle\psi^{(s + 1)}, u_1\rangle}{n} = \frac{g(n a_{+1}^{(s)}) + g(n a_{-1}^{(s)})}{2} + O(\|\delta^{(s)}\|_\infty),
	\]
	and
	\[
	\zeta^{(s+1)}_2 = \frac{\langle\psi^{(s + 1)}, u_2\rangle}{n} = \frac{g(n a_{+1}^{(s)}) - g(n a_{-1}^{(s)})}{2} + O(\|\delta^{(s)}\|_\infty).
	\]
	Note that $g(na_{\pm1}^{(s)}) = \bone_{\{a_{\pm 1}^{(s)} > 0\}} + O(\|\delta^{(s)}\|_\infty)$. Now, using \eqref{eq:psi_approx},we have
	\begin{align}\label{eq:lim_closeness}\nonumber
	&\frac{\|\psi^{(s+1)} - \ell(\psi^{(0)})\|_2^2}{n}\\ \nonumber &\qquad=\frac{\|(g(n a_{+1}^{(s)}) - \bone_{\{a_{+1}^{(0)} > 0\}}) \icone + (g(n a_{-1}^{(s)}) - \bone_{\{a_{-1}^{(0)} > 0\}}) \ictwo + \delta^{(s)}\|^2}{n}\\ \nonumber
	&\qquad\le \frac{2(\|(g(n a_{+1}^{(s)}) - \bone_{\{a_{+1}^{(0)} > 0\}}) \icone\|_2^2 + \|(g(n a_{-1}^{(s)}) - \bone_{\{a_{-1}^{(0)} > 0\}}) \ictwo \|_2^2 + \|\delta^{(s)}\|^2)}{n}\\ \nonumber
	&\qquad\le |g(n a_{+1}^{(s)}) - \bone_{\{a_{+1}^{(0)} > 0\}}|^2 + |g(n a_{-1}^{(s)}) - \bone_{\{a_{-1}^{(0)} > 0\}}|^2 + 2 \|\delta^{(s)}\|_{\infty}^2\\
	&\qquad= |\bone_{\{a_{+ 1}^{(s)} > 0\}} - \bone_{\{a_{+1}^{(0)} > 0\}}|^2 + |\bone_{\{a_{+ 1}^{(s)} > 0\}} - \bone_{\{a_{-1}^{(0)} > 0\}}|^2 + O(\|\delta^{(s)}\|_{\infty}^2).
	\end{align}
	From the above representation and our assumption on $n|a^{(0)}_{\pm 1}|$, the bound for $s = 1$ follows. We will now consider the four different cases of different signs of $a_{\pm 1}^{(s)}$.
	\vskip5pt
	\textbf{Case 1: $a_1^{(s)} > 0, a_{-1}^{(s)} > 0.$} In this case $g(na_1^{(s)}) = g(n a_{-1}^{(s)})= 1 + O(\|\delta^{(s)}\|_{\infty})$, so that 
	\[
	(\zeta_1^{(s+1)}, \zeta_2^{(s+1)}) = (1, 0) + O(\|\delta^{(s)}\|_{\infty}).
	\]
	This implies that
	\[
	a_{\pm 1}^{(s+1)} = 2\tstar\alpha_+ + O(\|\delta^{(s)}\|_{\infty}).
	\]
	If $\alpha_+ > 0$, $a_{\pm 1}^{(s+1)}$ have the same sign as $a_{\pm 1}^{(s)}$. Otherwise, if $\alpha_+ < 0$, both of them become negative (and we thus have to go to Case 2 below). Note that, here and in the subsequent cases, we are using that fact that $\|\delta^{(s)}\|_{\infty} = o(1)$, for $s = 0$, by our assumption and it stays the same for $s \ge 1$ because of relations like the above (that is $a_{\pm 1}^{(1)} = - 2\tstar \alpha_+ + o(1)$, so that $\|\delta^{(1)}\|_{\infty} = \exp(-n \min\{| a_{+1}^{(1)}|,  |a_{-1}^{(1)}|\}) = O(\exp(-C n\tstar\alpha_+)) = o(1)$, and so on). 
	\vskip5pt
	\textbf{Case 2: $a_1^{(s)} < 0, a_{-1}^{(s)} < 0.$} In this case $1 - g(na_1^{(s)}) = 1 - g(n a_{-1}^{(s)})= 1 + O(\|\delta^{(s)}\|_{\infty})$, so that 
	\[
	(\zeta_1^{(s+1)}, \zeta_2^{(s+1)}) = (0, 0) + O(\|\delta^{(s)}\|_{\infty}).
	\]
	This implies that
	\[
	a_{\pm 1}^{(s+1)} = -2\tstar\alpha_+ + O(\|\delta^{(s)}\|_{\infty}).
	\]
	If $\alpha_+ > 0$, $a_{\pm 1}^{(s+1)}$ have the same sign as $a_{\pm 1}^{(s)}$. Otherwise, if $\alpha_+ < 0$, both of them become positive (and we thus have to go to Case 1 above). 
	\vskip5pt
	\textbf{Case 3: $a_1^{(s)} > 0, a_{-1}^{(s)} < 0.$} In this case $g(na_1^{(s)}) = 1 - g(n a_{-1}^{(s)})= 1 + O(\|\delta^{(s)}\|_{\infty})$, so that 
	\[
	(\zeta_1^{(s+1)}, \zeta_2^{(s+1)}) = (\frac{1}{2}, \frac{1}{2}) + O(\|\delta^{(s)}\|_{\infty}).
	\]
	This implies that
	\[
	a_{\pm 1}^{(s+1)} = \pm 2\tstar\alpha_- + O(\|\delta^{(s)}\|_{\infty}).
	\]
	Since $\alpha_- > 0$, $a_{\pm 1}^{(s+1)}$ have the same sign as $a_{\pm 1}^{(s)}$. 
	\vskip5pt
	\textbf{Case 4: $a_1^{(s)} < 0, a_{-1}^{(s)} > 0.$} In this case $1 - g(na_1^{(s)}) = g(n a_{-1}^{(s)})= 1 + O(\|\delta^{(s)}\|_{\infty})$, so that 
	\[
	(\zeta_1^{(s+1)}, \zeta_2^{(s+1)}) = (\frac{1}{2}, -\frac{1}{2}) + O(\|\delta^{(s)}\|_{\infty}).
	\]
	This implies that
	\[
	a_{\pm 1}^{(s+1)} = \mp 2\tstar\alpha_- + O(\|\delta^{(s)}\|_{\infty}).
	\]
	Since $\alpha_- > 0$, $a_{\pm 1}^{(s+1)}$ have the same sign as $a_{\pm 1}^{(s)}$. 
	\vskip10pt
	
	Note that, in the case $\alpha_+ = 0$, $a_{\pm 1}^{(s)} = \pm 4 \tstar \zeta_2^{(s)} \alpha_-$, so that $a_{\pm 1}^{(s)}$ have opposite signs and we land in Cases 3 or 4.
	
	We conclude that, if $\alpha_+ \ge 0$, then we stay in the same case where we began, and otherwise if $\alpha_+ < 0$ we have a cycling behavior between Cases 1 and 2. Now the desired conclusion follows from the bound \eqref{eq:lim_closeness}.
	
In the proof above, we can allow sparser graphs, with $\pstar, \qstar \gg \frac{1}{n}$. More explicitly, let $\pstar = \rho_n a, \qstar = \rho_n b$, with $a > b > 0$ and $\rho_n \gg \frac{1}{n}$. Then, $\tstar = \Omega(1),$ and $\alpha_{+} \le \pstar - \qstar = \rho_n(a - b), \alpha_- = (\pstar - \qstar)/2 = \rho_n (a - b)/2$. So, we do have $n\tstar |\alpha_\pm| \rightarrow \infty$.   
\end{proof}

\begin{proof}[Proof of Theorem~\ref{thm:samp_known_pq}]
We begin by noting that $ A-\lambdastar(J-I)-M = A - \E (A|Z) := A-\tilde{P}$. For the first iteration, we rewrite the sample iterations \eqref{eq:bcavi_sample_known_pq} as
\begin{align*}
	\xi^{(1)} &= 4\tstar M\bigg(\psi^{(0)} - \frac{1}{2}\bone\bigg) + 4\tstar \underbrace{ (A - \tilde{P})\bigg(\psi^{(0)} - \frac{1}{2}\bone\bigg)}_{=: r^{(0)}}.
\end{align*}
Therefore, similar to the population case, we have
\begin{align*}
	\psi^{(1)}_i &= g(n a^{(0)}_{\sigma_i} + b^{(0)}_i + 4\tstar r^{(0)}_i).
\end{align*}
Note that
\begin{align}
\label{eq:rnoisedef}
	r^{(0)}_i = \sum_{j \ne i} (A_{ij} - \tilde{P}_{ij}) ( \psi^{(0)}_j - \frac{1}{2}).
\end{align}
Since our probability statements will be with respect to the randomness in $A$ and $\psi^{(0)}$ is independent of $A$, we may assume that $\psi^{(0)}$ is fixed. Let $Y_{ij} = (A_{ij} - \tilde{P}_{ij}) (\psi^{(0)}_j - \frac{1}{2})$. Then the $Y_{ij}$ are independent random variables for $j \ne i$, and $\E(Y_{ij}) = 0$. Also, $|Y_{ij}| \le |\psi^{(0)}_j - \frac{1}{2}| \le \|\psi^{(0)} - \frac{1}{2}\|_{\infty} = \Delta$, say, and $\E Y_{ij}^2 = (\psi^{(0)}_j - \frac{1}{2})^2 \mathrm{Var}(A_{ij}) = O(\rho_n (\psi^{(0)}_j - \frac{1}{2})^2)$. So, by Bernstein's inequality,
\begin{align}
	\P(\frac{1}{n}\sum_{j \ne i} Y_{ij} > \epsilon ) &\le \exp\bigg( \frac{-\frac{1}{2}n^2\epsilon^2}{\sum_{j \ne i} \E Y_{ij}^2 + \frac{1}{3}\Delta n \epsilon} \bigg)	\notag\\
	&\le \exp\bigg( \frac{-\frac{1}{2}n^2\epsilon^2}{C \rho_n \|\psi^{(0)} - \frac{1}{2}\|_{2}^2 + \frac{1}{3}\Delta n\epsilon} \bigg) \notag\\
	&\le  \exp\bigg( \frac{-\frac{1}{2}n^2\epsilon^2}{C n\rho_n \Delta^2  + \frac{1}{3}\Delta n \epsilon} \bigg).
	\label{eq:bernstein}
\end{align}
It follows from here that $ r^{(0)}_i = O(\sqrt{n\rho_n}\Delta \log n)$ with high probability, if $\sqrt{n\rho_n} = \Omega(\log n)$. In fact, by taking a suitably large constant in the big ``Oh'', we can show, via a union bound, that $\max_{i}  r^{(0)}_i = O(\sqrt{n\rho_n}\Delta \log n)$ with high probability.


Now, from our assumption $n|a^{(0)}_{\pm 1}| \gg \max \{\sqrt{n\rho_n}\| \psi^{(0)} - \frac{1}{2}\|_{\infty} \log n, 1 \}$, it follows that $n a^{(0)}_{\sigma_i} \gg 4\tstar r^{(0)}_i + b^{(0)}_i$ with high probability, simultaneously for all $i$.
Thus, similar to the population case, we can write
\[
	\psi^{(1)} = g(n a_{+1}^{(0)}) \bone_{\mathcal{C}_1} + g(n a_{-1}^{(0)}) \bone_{\mathcal{C}_2} + \hat{\delta}^{(0)},
\]
where $\|\hat{\delta}^{(0)}\|_{\infty} = O(\exp(-n\min\{|a_{+1}^{(0)}|, |a_{-1}^{(0)}|\})) = o(1)$, with high probability.  After this the proof proceeds like the the proof of Theorem~\ref{thm:pop_known_pq}, and so we omit it.

Let us consider the case with $s=2$ and we will show $r^{(1)}_i$ can be bounded in a general way. Now 
\begin{align*}
\xi^{(2)}
& = 4\tstar M(\psi^{(1)}-\frac{1}{2}\bone) + 4\tstar r^{(1)}	\\
& = 4\tstar M(\psi^{(1)}-\frac{1}{2}\bone) + \underbrace{4\tstar(A-\tilde{P})(\psi^{(1)}-\ell(\psi^{(0)}))}_{R_1} +  \underbrace{4\tstar(A-\tilde{P})(\ell(\psi^{(0)})-\frac{1}{2}\bone)}_{R_2}.	\\
\end{align*}
Now the analysis of the first term follows from Theorem~\ref{thm:pop_known_pq}. It is also easy to see $\max_i |R_{2,i}| = O_P(\sqrt{n\rho_n})$, since $\ell(\psi^{(0)}) \in \{\icone, \ictwo, \bone, \bzero, \frac{1}{2}\bone \}$. For $R_1$,
\begin{align*}
\max_i |R_{1,i}| & \leq \| R_1 \|_2 \leq C \| A-\tilde{P} \|_{op} \| \psi^{(1)}-\ell(\psi^{(0)}) \|_2		\\
 & = O_P(\sqrt{n\rho_n}) \sqrt{n} \cdot O (\exp(-\Theta(n\min\{|a_{+1}^{(0)}|, |a_{-1}^{(0)}|\}))) = o_P(1),
\end{align*}
under our assumption that $n|a^{(0)}_{\pm 1}| \gg \max \{\sqrt{n\rho_n}\| \psi^{(0)} - \frac{1}{2}\|_{\infty} \log n, 1 \}$. Hence $\max_i |r_i^{(1)}| = O_P(\sqrt{n\rho_n})$, and $n a^{(1)}_{\sigma_i} \gg 4\tstar r^{(1)}_i + b^{(1)}_i$ with high probability, simultaneously for all $i$. The same analysis as in the $s=1$ case follows. 

The case for general $s$ can be proved by induction using the same decomposition of $r^{(s)}$.
%
%
%
\end{proof}

\begin{proof}[Proof of Corollary~\ref{cor:volume_pop}]

From Theorem~\ref{thm:pop_known_pq}, it follows that, when $\alpha_+ > 0$,
\begin{align*}
	\mathfrak{M}(\mathcal{S}_{\bone}) &\ge \mathfrak{M}(\{\psi^{(0)} \mid  a_{+1}^{(0)} > 0, a_{-1}^{(0)} > 0, n a_{\pm 1}^{(0)} \gg 1\}\\
	&= \mathfrak{M}(\{\psi^{(0)} \mid  a_{+1}^{(0)} \gg \frac{1}{n}, a_{-1}^{(0)} \gg \frac{1}{n} \}) \\
	&\ge \mathfrak{M}(\{\psi^{(0)} \mid  a_{+1}^{(0)} > \frac{1}{n^\gamma}, a_{-1}^{(0)} > \frac{1}{n^\gamma} \}),
\end{align*}

for any $0 < \gamma < 1$ and so on for the other other limit points. 


More explicitly,
\begin{align*}
\{ \psi^{(0)} \mid a_{+1}^{(0)} > \frac{1}{n^\gamma}, a_{-1}^{(0)} > \frac{1}{n^\gamma}\} &= \{\psi^{(0)} \mid (\zeta_1^{(0)} - \frac{1}{2})\alpha_+ + \zeta_2^{(0)}\alpha_- > \frac{1}{4t n^\gamma}, \\
& \qquad\qquad\qquad\qquad\qquad (\zeta_1^{(0)} - \frac{1}{2})\alpha_+ - \zeta_2^{(0)}\alpha_- > \frac{1}{4t n^\gamma}\} \\
&= H_+^{\gamma} \cap H_-^{\gamma} \cap [0,1]^n,
\end{align*}

All in all, we have
\[
 	\mathfrak{M}(\mathcal{S}_{\bone}) \ge \lim_{\gamma \uparrow 1}\mathfrak{M}(H_+^{\gamma} \cap H_-^{\gamma} \cap [0,1]^n).
\]
This completes the proof.
\end{proof}

The main proof of Theorem~\ref{thm:halfcase} relies on a few lemmas, which we defer to the end of the proof. 
\begin{proof}[Proof of Theorem~\ref{thm:halfcase}]
	
For convenience, we assume $A$ has self loops, which has no effect on the conclusion. Similar to the notation used in the proof of Theorem~\ref{thm:samp_known_pq}, we decompose $\xi_i$ as the population update plus noise,
\begin{align}
\xi_i^{(s+1)} = 4\tstar \underbrace{M_{i,\cdot}(\psi^{(s)}-\frac{1}{2}\bone)}_{\text{signal}} + 4\tstar \underbrace{ (A-\E(A|Z))_{i,\cdot} (\psi^{(s)}-\frac{1}{2}\bone)}_{r^{(s)}_i}. 
	\label{eq:update_knownpq_signalnoise}
\end{align}
Note that the signal part is constant for $i\in\icone$ and $i\in\ictwo$. For convenience denote 
\begin{align}
s_1& =M_{i,\cdot}(\psi^{(0)}-\frac{1}{2}\bone), \qquad  i\in\icone	\notag\\
s_2& =M_{i,\cdot}(\psi^{(0)}-\frac{1}{2}\bone), \qquad  i\in\ictwo.
\label{eq:s_def}
\end{align}
Similarly, define $s_1^{(1)}$ and $s_2^{(1)}$ in terms of $\psi^{(1)}$. By Lemma~\ref{lem:general_lb_ub}, since $\pstar > \lambdastar > \qstar$, for $\Delta_1, \Delta_2>0$,
\begin{align}
s_1^{(1)} & = (\pstar-\lambdastar) \sum_{i \in\cC_1} (\psi_i^{(1)}-\frac{1}{2}) +  (\qstar-\lambdastar) \sum_{i \in\cC_2} (\psi_i^{(1)}-\frac{1}{2})	\notag\\
& \geq (\pstar-\lambdastar) \frac{n}{2}\left(\frac{1}{2}-\Phi\left(-\frac{s_1-\Delta_1}{\sigma_{\psi}}  \right)  \right) + (\qstar-\lambdastar) \frac{n}{2} \left( \frac{1}{2}-  \Phi\left(-\frac{s_2+\Delta_2}{\sigma_{\psi}}  \right) \right)		\notag\\
& \qquad\qquad  - \underbrace{ O(n\rho_n)(e^{-4\tstar\Delta_1} +e^{-4\tstar\Delta_2}) - O(n\rho_n)\frac{\rho_{\psi}}{\sigma_{\psi}^3}- O_P(\sqrt{n}\rho_n)}_{R_{\psi}}.
\label{eq:s1_1_lb}
\end{align}
Similarly, 
\begin{align}
s_2^{(1)} & = (\qstar-\lambdastar) \sum_{i \in\cC_1} (\psi_i^{(1)}-\frac{1}{2}) +  (\pstar-\lambdastar) \sum_{i \in\cC_2} (\psi_i^{(1)}-\frac{1}{2})	\notag\\
&\leq (\qstar-\lambdastar) \frac{n}{2}\left(\frac{1}{2}-\Phi\left(-\frac{s_1-\Delta_1}{\sigma_{\psi}}  \right)  \right) + (\pstar-\lambdastar) \frac{n}{2} \left( \frac{1}{2}-  \Phi\left(-\frac{s_2+\Delta_2}{\sigma_{\psi}}  \right) \right)	+  R_{\psi}	\notag\\
\label{eq:s2_1_ub}
\end{align}

We consider bounding $s_1^{(1)}$ and $s_2^{(1)}$ based on the signs of $s_1$ and $s_2$, which only depend on $\psi^{(0)}$. Therefore in each case, we first consider the conditional distribution given $\psi^{(0)}$.

{\bf{Case 1: $s_1>0$, $s_2<0$.}}

Let $\Delta_1=\epsilon s_1$, $\Delta_2=-\epsilon s_2$ for some small $\epsilon>0$. We have
\begin{align*}
\frac{1}{2}-\Phi\left(-\frac{(1-\epsilon)s_1}{\sigma_{\psi}}  \right) & \geq \frac{(1-\epsilon)s_1}{\sigma_{\psi}\sqrt{2\pi}} \exp\left( -\frac{(1-\epsilon)^2s_1^2}{2\sigma_{\psi}^2} \right),	\\
\Phi\left(-\frac{(1-\epsilon)s_2}{\sigma_{\psi}} \right) - \frac{1}{2} & \geq  -\frac{(1-\epsilon)s_2}{\sigma_{\psi}\sqrt{2\pi}} \exp\left( -\frac{(1-\epsilon)^2s_2^2}{2\sigma_{\psi}^2} \right),
\end{align*}
where we have used 
\begin{align}
|\Phi(x)-1/2| & = \frac{1}{\sqrt{2\pi}} \int_0^{|x|} e^{-u^2/2}du 	\notag\\ 
&  \geq \frac{|x|}{\sqrt{2\pi}}e^{-x^2/2}.
\label{eq:normal_cdf_lb}
\end{align}
Applying the above to~\eqref{eq:s1_1_lb},
\begin{align}
s_1^{(1)} & \geq \frac{n(1-\epsilon)}{2\sqrt{2\pi}\sigma_{\psi}} ((\pstar-\lambdastar)|s_1| + (\lambdastar-\qstar)|s_2| ) \exp\left( -\frac{(1-\epsilon)^2s_2^2\vee s_1^2}{2\sigma_{\psi}^2}  \right) - R_{\psi}.
\label{eq:s1_1_case1}
\end{align}
Similar arguments show
\begin{align}
s_2^{(1)} & \leq -\frac{n(1-\epsilon)}{2\sqrt{2\pi}\sigma_{\psi}} ((\lambdastar-\qstar)|s_1| + (\pstar-\lambdastar)|s_2| ) \exp\left( -\frac{(1-\epsilon)^2s_2^2\vee s_1^2}{2\sigma_{\psi}^2}  \right) + R_{\psi}	
\label{eq:s2_1_case1}
\end{align}

{\bf{Case 2: $s_1<0$, $s_2>0$.}}

The same analysis applies with the role of $\cC_1$ and $\cC_2$ interchanged. 

{\bf{Case 3: $s_1>0$, $s_2>0$.}} 

WLOG assume $s_1>s_2>0$. Taking $\Delta_1=\Delta_2=\epsilon (s_1-s_2)$,~\eqref{eq:s1_1_lb} becomes
\begin{align}
s_1^{(1)} & \geq \frac{n}{2\sqrt{2\pi}\sigma_{\psi}} [(\pstar-\lambdastar)(s_1-\epsilon(s_1-s_2)) - (\lambdastar-\qstar)(s_2+\epsilon(s_1-s_2)) ] \exp\left( -\frac{(1-\epsilon)^2 s_1^2}{2\sigma_{\psi}^2}  \right) - R_{\psi}		\notag\\
& \geq \frac{n}{2\sqrt{2\pi}\sigma_{\psi}} [(\lambdastar-\qstar) - \epsilon(\pstar-\qstar)]|s_1-s_2| \exp\left( -\frac{(1-\epsilon)^2  s_1^2}{2\sigma_{\psi}^2}  \right) - R_{\psi}	
\label{eq:s1_1_case3}
\end{align}
using $\pstar-\lambdastar > \lambdastar -\qstar$ (Proposition~\ref{prop:lambda_lbub}). Since $\lambdastar-\qstar=\Omega(\rho_n)$ also by Proposition~\ref{prop:lambda_lbub}, choose a $\epsilon$ small enough so that $(\lambdastar-\qstar) - \epsilon(\pstar-\qstar) \geq \Omega(\rho_n)$. 

Similarly, taking $\Delta_1=\epsilon_n s_1$, $\Delta_2=\epsilon_n s_2$,
\begin{align}
s_2^{(1)} & \leq  -\frac{n}{2\sqrt{2\pi}\sigma_{\psi}} [(\lambdastar-\qstar)(1-\epsilon_n)s_1 -(\pstar-\lambdastar)(1+\epsilon_n)s_2] \exp\left( -\frac{(1+\epsilon_n)^2 s_1^2}{2\sigma_{\psi}^2}  \right) + R_{\psi}	
\label{eq:s2_1_case3},
\end{align}
Letting $\epsilon_n\to 0 $ slowly and denote $c=\frac{(\lambdastar-\qstar)(1-\epsilon_n)}{(\pstar-\lambdastar)(1+\epsilon_n)}-\eta$, for some small $\eta>0$. When $s_2\leq c s_1$, 
\begin{align}
s_2^{(1)} & \leq - \frac{n}{2\sqrt{2\pi}\sigma_{\psi}}  \eta (\pstar-\lambdastar) |s_1| + R_{\psi}.
\end{align}  
By Lemma~\ref{lem:s_bounds},  $s_2\leq c s_1$ happens with probability
\begin{align*}
P(0<s_2\leq cs_1) =  \frac{\arctan(c_u)}{2\pi} -  \frac{\arctan(c_{\ell})}{2\pi} + O(n^{-1/2}),
\end{align*}
where $c_u = \frac{\pstar-\lambdastar}{\lambdastar-\qstar} $, $c_{\ell} = \frac{(\pstar-\lambdastar)+c(\lambdastar-\qstar)}{c(\pstar-\lambdastar)+(\lambdastar-\qstar)}$.

When $s_2>s_1>0$, the analysis is the same by symmetry. We have the same bounds for $s_1^{(1)}$ and $s_2^{(1)}$ with $s_1$ and $s_2$ interchanged. By a similar calculation, we need
\begin{align*}
P(0<s_1\leq cs_2) =  \frac{\arctan(c_{\ell}^{-1})}{2\pi} -  \frac{\arctan(c_u^{-1})}{2\pi} + O(n^{-1/2}),
\end{align*}

{\bf{Case 4: $s_1<0$, $s_2<0$.}} 
By symmetry, $g(4\tstar(s_1+r_i^{(0)}))-\frac{1}{2} = \frac{1}{2}-g(-4\tstar(s_1+r_i^{(0)}))$ (similarly for $g(4\tstar(s_2+r_i^{(0)}))$). It suffices to apply the same analysis in Case 3 to $-s_1, -s_2$ and $-r_i^{(0)}$. For example, when $s_1<s_2<0$, $-s_1^{(1)}$ is lower bounded by~\eqref{eq:s1_1_case3}, $-s_2^{(1)}$ is upper bounded by~\eqref{eq:s2_1_case3} when $0<-s_1 <-cs_2$.

Now combining all the cases, define event $B$ as 
\begin{align*}
B=\left\{|s_1^{(1)}|, |s_2^{(1)}| \geq Cn\rho_n \sigma_{\psi}^{-1}\min\{|s_1|, |s_2|, |s_1-s_2|\} \exp\left( -\frac{(1+\epsilon)^2s_2^2\vee s_1^2}{2\sigma_{\psi}^2}  \right) - R_{\psi}, s_1^{(1)}s_2^{(1)} <0	  \right\}.
\end{align*}
Cases 1--4 imply 
\begin{align}
P( B ) & = \sum_{\psi: s_1s_2<0}P(B|\psi^{(0)}=\psi)P(\psi^{(0)}=\psi) + \sum_{\psi: s_1s_2>0}P(B|\psi^{(0)}=\psi)P(\psi^{(0)}=\psi)	\notag\\
& \geq P(s_1s_2<0) + 2P(0<s_2<cs_1) + 2P(0<s_1 <cs_2)	\notag\\
& = \frac{1}{2} + \frac{2\arctan(c_u^{-1})}{\pi} + \frac{\arctan(c_u)-\arctan(c_u^{-1})}{\pi} - \frac{\arctan(c_{\ell})-\arctan(c_{\ell}^{-1})}{\pi}	\notag\\
& = 1-\frac{\arctan(c_{\ell})-\arctan(c_{\ell}^{-1})}{\pi},
\label{eq:pB_lb}
\end{align}
where 
\begin{align*}
P(s_1>0, s_2<0) = P(s_1<0, s_2>0) = \frac{1}{4} + \frac{\arctan(c_u^{-1})}{\pi} 
\end{align*}
using calculations similar to Lemma~\ref{lem:s_bounds}.

We note that $|s_1|$, $|s_2|$, $|s_1-s_2|$ are of order $\Omega_P(\rho_n\sqrt{n})$ Lemma~\ref{lem:anti_conc}. Also $\sigma^2_{\psi}= O_P(n\rho_n)$, $\rho_{\psi}= O_P(n\rho_n)$, $e^{-4\tstar|s_1|}, e^{-4\tstar|s_2|} = O_P(\exp(-\rho_n \sqrt{n}))$, $R_{\psi} = o_P(n\rho^{3/2})$.
It follows $s_1^{(1)} \geq \Omega_P(n\rho_n^{3/2})$, $s_2^{(1)} \leq -\Omega_P(n\rho_n^{3/2})$, and by~\eqref{eq:pB_lb},
\begin{align}
P(|s_1^{(1)}|, |s_2^{(1)}| \geq \Omega(n\rho_n^{3/2}),  s_1^{(1)}s_2^{(1)} <0	) \geq 1-\frac{\arctan(c_{\ell})-\arctan(c_{\ell}^{-1})}{\pi}.
\end{align}

In the next iteration, write the true labels as $z_0=\icone \ind\{s_1^{(1)}>0 \} + \ictwo \ind\{s_1^{(1)}<0\}$. When $s_1^{(1)} >0$ holds,
\begin{align}
|\psi_i^{(2)} - z_{0,i}| = \frac{1}{1+e^{\sigma_i \xi_i^{(2)}}} \leq e^{-x_0} + \ind\{ \sigma_i \xi_i^{(2)} \leq x_0 \}
\label{eq:psi2_z}
\end{align}
for any $x_0 > 0$. For $i\in\cC_1$,
\begin{align*}
\xi_i^{(2)} & = 4\tstar s_1^{(1)} + 4\tstar r_i^{(1)} \\
& = 4\tstar s_1^{(1)} + 4\tstar(A-P)_{i,\cdot}(z_0-\frac{1}{2}\bone)  + 4\tstar(A-P)_{i,\cdot}(\psi^{(1)}-z_0)
\end{align*} 

Taking $x_0 = 4\tstar\rho_n^{3/2}n/\sqrt{c_n}$ for some $c_n\to\infty$ slowly, using the fact that $s_1^{(1)}\geq\Omega_P(n\rho_n^{3/2})$, $4\tstar s_1^{(1)}> 3x_0$ for large $n$ with high probability, $(A-P)_{i,\cdot}(z_0-\frac{1}{2}\bone) = O_P(\sqrt{n\rho_n}\log n)$ uniformly for all $i$, 
\begin{align}
\ind\{ \xi_i^{(2)} \leq x_0 \} & \leq
\ind\left\{ 4\tstar s_1^{(1)}- O_P(\sqrt{n\rho_n}\log n)) \leq 2x_0\right\}  \notag\\
& \qquad + \ind\left\{ 4\tstar(A-P)_{i,\cdot} (\psi^{(1)} -z_0) \leq -x_0\right\} \notag\\
& = \exp\left( 2x_0 - 4\tstar s_1^{(1)} + O_P(\sqrt{n\rho_n}\log n)\right) + \ind\left\{ (A-P)_{i,\cdot} (\psi^{(1)} -z_0) \leq -\rho_n^{3/2}n/\sqrt{c_n} \right\}	\notag\\
& = \exp(-\rho_n^{3/2}n/\sqrt{c_n})+ \ind\left\{ (A-P)_{i,\cdot} (\psi^{(1)} -z_0) \leq -\rho_n^{3/2}n/\sqrt{c_n} \right\}
\label{eq:xi2_c1}
\end{align}
with high probability. Similarly for $i\in\cC_2$, since $s_2^{(1)} \leq - \Omega_P(n\rho_n^{3/2})$,
\begin{align}
\ind\{ -\xi_i^{(2)} \leq x_0 \} & \leq \exp(-\rho_n^{3/2}n/\sqrt{c_n}) + \ind\left\{(A-P)_{i,\cdot}(\psi^{(1)} -z_0) \geq \rho_n^{3/2}n/\sqrt{c_n}\right\}.
\label{eq:xi2_c2}
\end{align}
Summing~\eqref{eq:psi2_z} using~\eqref{eq:xi2_c1} and~\eqref{eq:xi2_c2},
\begin{align}
\| \psi^{(2)} - z_0\|_1 & \leq n \exp(-\rho_n^{3/2}n/\sqrt{c_n}) + \sum_{i} \ind\left\{\left|  (A-P)_{i,\cdot}(\psi^{(1)} -z_0)\right| \geq \rho_n^{3/2}n/\sqrt{c_n} \right\} \notag\\
& \leq n \exp(-\rho_n^{3/2}n/\sqrt{c_n})  + \frac{C (\psi^{(1)} -z_0)^T (A-P)^2 (\psi^{(1)} -z_0) c_n}{n^2\rho_n^3} \notag\\
& \leq n \exp(-\rho_n^{3/2}n/\sqrt{c_n}) + \frac{C \| A-P \|_{op}^2 \|\psi^{(1)} -z_0 \|_2^2 c_n}{n^2\rho_n^3}    \notag\\
& \leq n\exp(-\rho_n^{3/2}n/\sqrt{c_n}) + \frac{c_n}{n\rho_n^2} \|\psi^{(1)} -z_0 \|_1,
\label{eq:contraction}
\end{align}
with high probability, where we have used the fact that there exist $C_1, \epsilon>0$ such that $\|A-P\|_{op} \leq C_1\sqrt{n\rho_n}$ with probability at least $1-n^{-\epsilon}$ (Theorem 5.2 in~\cite{lei2015consistency}).

The case for $s_1^{(1)}<0$ is similar with $z_0=\ictwo$. 

For later iterations, note that when $z_0=\icone$, $\| \psi^{(2)} - z_0\|_1 = n/2-\langle \psi^{(2)}, u_2\rangle$, then~\eqref{eq:contraction} implies 
\begin{align*}
\langle \psi^{(2)}, u_2\rangle \geq \frac{n}{2} - \epsilon_n n
\end{align*}
for some $\epsilon_n=o_P(1)$, and 
\begin{align*}
\sum_{i\in\cC_1} (\psi_i^{(2)} - 1/2) \geq \frac{n}{4} - \epsilon_n n.
\end{align*}
Then since $\pstar+\qstar-2\lambdastar>0$,
\begin{align*}
s_1^{(2)} = (\lambdastar-\qstar)\langle \psi^{(2)}, u_2 \rangle + (\pstar+\qstar-2\lambdastar)\sum_{i\in\cC_1}(\psi_i^{(2)}-1/2) \geq \Omega_P(n\rho_n),
\end{align*}
and similarly, 
\begin{align*}
\sum_{i\in\cC_2} (\psi_i^{(2)} - 1/2) \leq -\frac{n}{4} + \epsilon_n n,
\end{align*}
\begin{align*}
s_2^{(2)} = -(\lambdastar-\qstar)\langle \psi^{(2)}, u_2 \rangle + (\pstar+\qstar-2\lambdastar)\sum_{i\in\cC_2}(\psi_i^{(2)}-1/2) \leq -\Omega_P(n\rho_n).
\end{align*}
The rest of the argument applies from~\eqref{eq:psi2_z}-\eqref{eq:contraction} with a larger rate for $s_1^{(2)}$ and $s_2^{(2)}$, which will give the contraction 
\begin{align}
\| \psi^{(3)} - z_0\|_1 & \leq n\exp(-\rho_n n/\sqrt{c_n}) + \frac{c_n}{n\rho_n} \|\psi^{(2)} -z_0 \|_1.
\end{align}

The arguments can be repeated for all the later iterations.

\end{proof}

Now we state and prove all the lemmas needed in the main proof. First we have a few concentration lemmas.
\begin{lemma}[Berry-Esseen bound]
	\label{lem:be_bound}
	\begin{align*}
	\sup_{x\in\mathbb{R}}|P\left(r_i^{(0)}/\sigma_{\psi} \leq x \mid \psi^{(0)} \right) - \Phi(x)| \leq C_0 \cdot \frac{\rho_{\psi}}{\sigma^3_{\psi}},
	\end{align*}
	where $C_0$ is a general constant, $\rho_{\psi}$ and $\sigma_{\psi}$ depend on $\psi^{(0)}$. 
\end{lemma}

\begin{proof}
	Define 
	\begin{align*}
	\sigma^2_{\psi} & :=\pstar(1-\pstar)\sum_{i\in\cC_1}(\psi_i^{(0)}-1/2)^2 + \qstar(1-\qstar)\sum_{i\in\cC_2}(\psi_i^{(0)}-1/2)^2,      \\
	\rho_{\psi} & :=\pstar(1-\pstar)(1-2\pstar+2\pstar^2)\sum_{i\in\cC_1}|\psi_i^{(0)}-1/2|^3 + \qstar(1-\qstar)(1-2\qstar+2\qstar^2)\sum_{i\in\cC_2}|\psi_i^{(0)}-1/2|^3.
	\end{align*}
	It follows by the Berry-Esseen bound that
	\begin{align*}
	\sup_{x\in\mathbb{R}}|P\left(r_i^{(0)}/\sigma_{\psi} \leq x \mid \psi^{(0)} \right) - \Phi(x)| \leq C_0 \cdot \frac{\rho_{\psi}}{\sigma^3_{\psi}} 
	\end{align*}
	for some general constant $C_0$, where $\Phi$ is the CDF of standard Gaussian.
\end{proof}

\begin{lemma}[Littlewood-Offord]
	\label{lem:anti_conc}
	Let $s_1 = (\pstar-\lambdastar)\sum_{i\in\cC_1} (\psi_i^{(0)}-1/2)+(\qstar-\lambdastar)\sum_{i\in\cC_2} (\psi_i^{(0)}-1/2)$, $s_2 = (\qstar-\lambdastar)\sum_{i\in\cC_1} (\psi_i^{(0)}-1/2)+(\pstar-\lambdastar)\sum_{i\in\cC_2} (\psi_i^{(0)}-1/2)$.
	Then
	$$P\left( |s_1| \leq c \right) \leq B\cdot \frac{c}{\rho_n\sqrt{n}}$$
	for $c>0$. The same bound holds for $ |s_2|, |s_1-s_2|$.
\end{lemma}
\begin{proof}
	Noting that $2\psi_i^{(0)}-1\in \{-1,1\}$ each with probability $1/2$, and $\qstar < \lambdastar < \pstar$, this is a direct consequence of the Littlewood-Offord bound in~\cite{erdos1945lemma}.
\end{proof}

\begin{lemma}[McDiarmid's Inequality]
	\label{lem:mcd_bound}
	Recall $r_i^{(0)} = (A-\E(A|Z))(\psi^{(0)}-\frac{1}{2}\bone)$ and let $h(r_i^{(0)})$ be a bounded function with $\|h \|_{\infty} \leq M$. Then 
	$$P\left( \left| \frac{2}{n}\sum_{i\in\cC_1} h(r_i^{(0)}) - \E(h(r_i^{(0)})|\psi^{(0)}) \right| > w \mid \psi^{(0)} \right) \leq \exp\left(-\frac{w^2}{nM} \right).$$ The same bound holds for $i\in\cC_2$.
\end{lemma}

\begin{proof}
	Define $\phi=\frac{2}{n}\sum_{i\in\cC_1} h(r_i^{(0)})$, then conditional on $\psi^{(0)}$, $\phi$ is only a function of $(A_{ij})_{i<j, i\in\cC_1}$. Replacing any $A_{ij}$ with $A'_{ij}\in\{0,1\}$,
	\begin{align*}
	|\phi(A_{12}, \dots, A_{ij},\dots) - \phi(A_{12}, \dots, A'_{ij},\dots)| \leq \frac{8M}{n}.
	\end{align*}
	and 
	\begin{align*}
	\sum_{i<j, i\in\cC_1} |\phi(A_{12}, \dots, A_{ij},\dots) - \phi(A_{12}, \dots, A'_{ij},\dots)|\leq 2nM 
	\end{align*}
	The desired bound follows by McDiarmid's inequality.
\end{proof}

Using the normal approximation, we can also derive the following probability bound for $s_1$ and $s_2$.
\begin{lemma}
	For some constant $0<c<1$, 
	\begin{align*}
	P(0\leq s_2\leq  cs_1) = \frac{\arctan(c_u)}{2\pi} -  \frac{\arctan(c_{\ell})}{2\pi} + O(n^{-1/2}),
	\end{align*}
	where $c_{\ell} = \frac{(\pstar-\lambdastar)+c(\lambdastar-\qstar)}{c(\pstar-\lambdastar)+(\lambdastar-\qstar)}$, $c_u=\frac{\pstar-\lambdastar}{\lambdastar-\qstar}$.
	\label{lem:s_bounds}	
\end{lemma}

\begin{proof}
	For convenience, denote $T_1=\sum_{i\in\cC_1}(\psi_i^{(0)}-1/2)$, $T_2=\sum_{i\in\cC_2}(\psi_i^{(0)}-1/2)$, then
	\begin{align*}
	\{0\leq s_2\leq cs_1\} & = \left\{\frac{(\pstar-\lambdastar)+c(\lambdastar-\qstar)}{c(\pstar-\lambdastar)+(\lambdastar-\qstar)} T_2 \leq T_1 \leq \frac{\pstar-\lambdastar}{\lambdastar-\qstar} T_2 \right\}	\\
	& := \{c_{\ell} T_2 \leq T_1 \leq c_u T_2 \text{ and } T_1, T_2>0\}
	\end{align*}
	where $1<c_{\ell} <c_u$. It is easy to see that $\E(T_1)=\E(T_2)=0$, $\sigma^2_T:=\E(T_1^2) = \E(T_1^2) \asymp \rho_n^2 n$, $\E|T_1|^3 = \E|T_1|^3 \asymp \rho_n^3 n$. Then
	\begin{align}
	P(0\leq s_2\leq  cs_1) & = P(0\leq T_1\leq c_u T_2) - P(0\leq T_1 <  c_{\ell} T_2).
	\label{eq:s_to_t}
	\end{align}
	The first part can be calculated as 
	\begin{align*}
	P(0\leq T_1\leq c_u T_2) & = \sum_{t\geq 0}P(0\leq T_1\leq c_u t | T_2=t)	P(T_2=t)	\\
	& = \sum_{t\geq 0} P(0\leq Z_1\leq c_u T_2 \sigma_T^{-1} | T_2=t)P(T_2=t)	+ O(n^{-1/2})	\\
	& =  \E \left( (\Phi( c_u T_2 \sigma_T^{-1} )-1/2)\ind(T_2\geq 0) \right)+ O(n^{-1/2})
	\end{align*}
	using the Berry-Esseen bound, $Z_1\sim N(0,1)$. Now note that $(\Phi( c_u T_2 \sigma_T^{-1} )-1/2)\ind(T_2\geq 0)$ is continuous and monotonic in $T_2$. For every $t\in(0,1]$, there exists $a(t)> 0$ such that $\Phi( c_u T_2  \sigma_T^{-1} ) -1/2 \geq t \Leftrightarrow T_2 \sigma_T^{-1} \geq a(t) $. We have
	\begin{align*}
	\E \left( (\Phi( c_u T_2 \sigma_T^{-1} )-1/2)\ind(T_2\geq 0) \right)  & = \int_0^1 P\left( (\Phi( c_u T_2 \sigma_T^{-1} )-1/2)\ind(T_2\geq 0) \geq t \right) dt  \\
	& = \int_0^1 P(T_2\sigma_T^{-1} \geq a(t) ) dt	  \\
	& = \int_0^1 P(Z_2 \geq a(t) ) dt + O(n^{-1/2})		\\
	& = \E\left( (\Phi( c_u Z_2  )-1/2)\ind(Z_2\geq 0) \right)  + O(n^{-1/2}),	
	\end{align*}
	$Z_2\sim N(0,1)$, independent of $Z_1$. It remains to calculate the expectation, which can be written as 
	
	\begin{align*}
	w(x)&=\frac{1}{2\pi}\int_0^\infty \int_{0}^{x z}\exp(-u^2/2)\exp(-z^2/2)dudz\\
	\end{align*}
	for $x=c_u$. Now 
	\begin{align*}
	w'(x)&=\frac{1}{2\pi}\int_{0}^\infty z\exp(-(1+x^2)z^2/2)dz=\frac{1}{2\pi(1+x^2)}
	\end{align*}
	Integrating both sides, we get:
	$w(x)=\frac{\arctan(x)}{2\pi}+C$, where $C=0$ since $w(0)=0$. Thus $w(c_u) = \frac{\arctan(c_u)}{2\pi}$. The same calculation can be done for $P(0\leq T_1 \leq c_{\ell} T_2)$. Substituting into~\eqref{eq:s_to_t},
	\begin{align*}
	P(0\leq s_2 \leq cs_1) =  \frac{\arctan(c_u)}{2\pi} -  \frac{\arctan(c_{\ell})}{2\pi} + O(n^{-1/2})
	\end{align*}
	
\end{proof}

Finally, we have the following general bounds for $\sum_{i \in\cC_1} \psi_i^{(1)}$ and $\sum_{i \in\cC_2} \psi_i^{(1)}$.
\begin{lemma}
	For any $\Delta_1>0$,	
	\begin{align}
	\sum_{i \in\cC_1} \psi_i^{(1)}  & \geq  \frac{n}{2}\left(1-\Phi\left(-\frac{s_1-\Delta_1}{\sigma_{\psi}}  \right)  \right) - \frac{n}{2}e^{-4\tstar\Delta_1} -  C'n\cdot \frac{\rho_{\psi}}{\sigma_{\psi}^3}-O_P(\sqrt{n}),	\notag\\
	\sum_{i \in\cC_1} \psi_i^{(1)}  & \leq  \frac{n}{2} \left( 1-  \Phi\left(-\frac{s_1+\Delta_1}{\sigma_{\psi}}  \right) \right) + \frac{n}{2}e^{-4\tstar\Delta_1}
	+ C'n\cdot \frac{\rho_{\psi}}{\sigma_{\psi}^3}+ O_P(\sqrt{n}), 
	\end{align}
	where $\Phi$ is the CDF of standard Gaussian, $\rho_{\psi}$ and $\sigma_{\psi}$ are constants depending on $\psi^{(0)}$ defined in Lemma~\ref{lem:be_bound}, and the $O_P(\sqrt{n})$ terms are uniform for $\psi^{(0)}$. The same upper and lower bound hold for $i\in \cC_2$ and $s_2$.
	\label{lem:general_lb_ub}	
\end{lemma}

\begin{proof}
	Define an index set $J_1^+=\{i: r_i^{(0)} > -s_1+\Delta_1\}$, $\Delta_1>0$. Then for $i \in \cC_1\cap J_1^+$, 
	\begin{align}
	\psi_i^{(1)} & = g(4\tstar(s_1+r_i^{(0)})) \geq g(4\tstar \Delta_1)	\geq 1-e^{-4\tstar\Delta_1}. \notag
	\end{align}
	It follows then 
	\begin{align}
	\sum_{i \in\cC_1} \psi_i^{(1)}& \geq |\cC_1\cap J_1^+| (1-e^{-4\tstar\Delta_1})
	\label{eq:c1_lb_temp}
	\end{align}
	To calculate the size of the set, note that 
	\begin{align*}
	|\cC_1\cap J_1^+ |= \sum_{i\in\cC_1} \ind(r_i^{(0)}>-s_1+\Delta_1),
	\end{align*}
	By Lemma~\ref{lem:mcd_bound},
	\begin{align}
	|\cC_1\cap J_1| & = \frac{n}{2}P(r_i^{(0)}>-s_1+\Delta_1) \mid \psi^{(0)})+O_P(\sqrt{n})    \notag\\
	& \geq \frac{n}{2}\left(P(r>-s_1+\Delta_1 ) - C_0\cdot \frac{\rho_{\psi}}{\sigma_{\psi}^3}\right)-O_P(\sqrt{n}) \notag\\
	& = \frac{n}{2}\left(1-\Phi\left(-\frac{s_1-\Delta_1}{\sigma_{\psi}}  \right) \right) -  C'n\cdot \frac{\rho_{\psi}}{\sigma_{\psi}^3}-O_P(\sqrt{n}),
	\label{eq:size_c1j+}
	\end{align}
	where the second line follows from Lemma~\ref{lem:be_bound}, with $\Phi$ as the CDF of standard Gaussian, $r\sim N(0, \sigma^2_{\psi})$ and the $O_P(\sqrt{n})$ can be made uniform over $\psi^{(0)}$. \eqref{eq:c1_lb_temp} and~\eqref{eq:size_c1j+} imply
	\begin{align}
	\sum_{i \in\cC_1} \psi_i^{(1)}& \geq  \frac{n}{2}\left(1- \Phi\left(-\frac{s_1-\Delta_1}{\sigma_{\psi}}  \right)  \right) (1-e^{-4\tstar\Delta_1}) \notag\\
	& \qquad -  C'n\cdot \frac{\rho_{\psi}}{\sigma_{\psi}^3}-O_P(\sqrt{n})	  \notag\\
	& \geq \frac{n}{2}\left(1-\Phi\left(-\frac{s_1-\Delta_1}{\sigma_{\psi}}  \right)  \right) - \frac{n}{2}e^{-4\tstar\Delta_1}  	\notag\\
	& \qquad -  C'n\cdot \frac{\rho_{\psi}}{\sigma_{\psi}^3}-O_P(\sqrt{n}).	 
	\label{eq:c1_lb}
	\end{align}
	
	Similarly let $J_1^-
	=\{i: r_i^{(0)} < -s_1-\Delta_1\}$, $\Delta_1>0$. For $i\in\cC_1\cap J_1^-$, 
	\begin{align}
	\psi_i^{(1)} & = g(4\tstar(s_1+r_i^{(0)})) \leq g(-4\tstar \Delta_1) \leq e^{-4\tstar \Delta_1}.	\notag
	\end{align}
	We have 
	\begin{align}
	\sum_{i \in\cC_1} \psi_i^{(1)}& \leq |\cC_1\cap J_1^-| e^{-4\tstar \Delta_1} + \frac{n}{2} - |\cC_1\cap J_1^-|	\notag\\
	&=  \frac{n}{2} - |\cC_1\cap J_1^-|(1-e^{-4\tstar \Delta_1}),
	\label{eq:c1_ub_temp}
	\end{align}
	where 
	\begin{align}
	|\cC_1\cap J_1^-| & = \frac{n}{2}P(r_i^{(0)}<-s_1-\Delta_1 ) \mid \psi^{(0)})+O_P(\sqrt{n})    \notag\\
	& \geq  \frac{n}{2} \Phi\left(-\frac{s_1+\Delta_1}{\sigma_{\psi}}  \right) -  C'n\cdot \frac{\rho_{\psi}}{\sigma_{\psi}^3}-O_P(\sqrt{n}).
	\label{eq:size_c1j-}
	\end{align}
	\eqref{eq:c1_ub_temp} and~\eqref{eq:size_c1j-} give
	\begin{align}
	\sum_{i \in\cC_1} \psi_i^{(1)}& \leq \frac{n}{2} - \frac{n}{2}\Phi\left( -\frac{s_1+\Delta_1}{\sigma_{\psi}}  \right)(1-e^{-4\tstar\Delta_1})	\notag\\
	& \qquad   + C'n\cdot \frac{\rho_{\psi}}{\sigma_{\psi}^3}+ O_P(\sqrt{n})	\notag\\
	& \leq  \frac{n}{2} \left( 1-  \Phi\left(-\frac{s_1+\Delta_1}{\sigma_{\psi}}  \right) \right) + \frac{n}{2}e^{-4\tstar\Delta_1}	\notag\\
	& \qquad + C'n\cdot \frac{\rho_{\psi}}{\sigma_{\psi}^3}+ O_P(\sqrt{n}).
	\label{eq:c1_ub}
	\end{align}
	
\end{proof}

\begin{proof}[Proof of Corollary~\ref{cor:pq_noise}]

Let $\hat{t}, \hat{\lambda}$ be constants defined in the usual way in terms of $\hat{p}, \hat{q}$. First we observe using $\hat{p}, \hat{q}$ only replaces $t, \lambda$ with $\hat{t}, \hat{\lambda}$ everywhere in~\eqref{eq:update_knownpq_signalnoise}. Now 
\begin{align*}
    \hat{s}_1 & = (\pstar-\hat{\lambda})\sum_{i\in\cC_1} (\psi_i^{(0)}-1/2)+(\qstar-\hat{\lambda})\sum_{i\in\cC_2} (\psi_i^{(0)}-1/2)  \\
    \hat{s}_2 & = (\qstar-\hat{\lambda})\sum_{i\in\cC_1} (\psi_i^{(0)}-1/2)+(\pstar-\hat{\lambda})\sum_{i\in\cC_2} (\psi_i^{(0)}-1/2)
\end{align*}
We can check the rest of the analysis remains unchanged as long as 
\begin{enumerate}
\item $\frac{\pstar+\qstar}{2} > \hat{\lambda}$, 
\item $\hat{\lambda} - \qstar = \Omega(\rho_n) >0$.
\end{enumerate}

\end{proof}
\bk

\subsection{Proofs of results in Section~\ref{subsec:unknownpq}}

\begin{proof}[Proof of Proposition~\ref{prop:local_max}]
	That the described point is a stationary point is easy to verify, because of the presence of the $(\psi_i - \frac{1}{2})$ terms in the stationarity equations \eqref{eq:first_derivative}. Now, from \eqref{eq:second_derivative}, we see that the Hessian matrix at $ (\frac{1}{2}\bone, \frac{\bone^\top A \bone}{n(n - 1)}, \frac{\bone^\top A \bone}{n(n - 1)}, \frac{1}{2})$ is given by
	\[
	H = \begin{pmatrix}
	-4I & \bzero & \bzero \\
	\bzero^\top &  -\frac{n(n - 1)}{4 \hat{a} (1 - \hat{a})} & 0\\
	\bzero^\top & 0 &  -\frac{n(n - 1)}{4 \hat{a} (1 - \hat{a})}
	\end{pmatrix},
	\]
	where $\hat{a} = \frac{\bone^\top A \bone}{n(n - 1)}.$ Clearly, $H$ is negative definite. This completes the proof.
\end{proof}


\begin{proof}[Proof of Lemma~\ref{lem:unknownpupdate}]
	First note that conditioning on the true labels $Z$, $\E(A|Z)=\tilde{P}$. For notation simplicity, we omit the superscript of $\psi^{(0)}$. For the update of $\phat$, we have
	\ba{
		\phat=&\frac{\psi^T\tilde{P}\psi+(\bone-\psi)^T\tilde{P}(\bone-\psi)}{\psi^T(J-I)\psi+(\bone-\psi)^T(J-I)(\bone-\psi)}\nonumber\\
		 &\qquad\qquad+ \frac{\psi^T(A-\tilde{P})\psi+(\bone-\psi)^T(A-\tilde{P})(\bone-\psi)}{\psi^T(J-I)\psi+(\bone-\psi)^T(J-I)(\bone-\psi)},	
	\label{eq:p1_update} 
}
	where the first term can be written as 
	\bas{
		& \frac{\psi^T(\frac{\pstar+\qstar}{2}u_1u_1^T+\frac{\pstar-\qstar}{2}u_2u_2^T-\pstar I)\psi+(\bone-\psi)^T(\frac{\pstar+\qstar}{2}u_1u_1^T+\frac{\pstar-\qstar}{2}u_2u_2^T-\pstar I)(\bone-\psi)}{\psi^T(u_1u_1^T-I)\psi+(\bone-\psi)^T(u_1u_1^T-I)(\bone-\psi)}\\
		=& \frac{\frac{\pstar+\qstar}{2}n^2(\zeta_1^2 + (1-\zeta_1)^2)+n^2(\pstar-\qstar)\zeta_2^2-\pstar x}{\zeta_1^2 n^2 + (1-\zeta_1)^2 n^2-x}\\
		=& \frac{\pstar+\qstar}{2}+\frac{(\pstar-\qstar)(\zeta_2^2-x/2n^2)}{\zeta_1^2+(1-\zeta_1)^2-x/n^2},
	}
where $x=\psi^T\psi+(\bone-\psi)^T(\bone-\psi)\geq n/4$.
	The second term can be bounded by noting $\E( \psi^T(A-\tilde{P})\psi )=0$ and $\text{Var}( \psi^T(A-\tilde{P})\psi) \leq 2 n(n-1)\pstar$. By Chebyshev's inequality, $ \psi^T(A-\tilde{P})\psi = O_P(\sqrt{\rho_n}n)$. 
	
	This is because
	\begin{align*}
	\E_{\psi,A}[\psi^T(A-\tilde{P})\psi]=\E_{\psi}\E_A[\left.\psi^T(A-\tilde{P})\psi\right|\psi]=0,
	\end{align*}
	and
	\begin{align*}
	\var_{\psi,A}[\psi^T(A-\tilde{P})\psi]&=\E\var (\left.\psi^T(A-\tilde{P})\psi\right|\psi)+\var(\E[\left.\psi^T(A-\tilde{P})\psi\right|\psi])\\
	&=\E\var (\left.\psi^T(A-\tilde{P})\psi\right|\psi)\\
	&=4\E \sum_{i<j}\psi_i\psi_j  \var(A_{ij})\leq 2n(n-1)\pstar.
	\end{align*}


	$(1-\psi)^T(A-\tilde{P})(1-\psi)$ can be handled similarly, and 
	\bas{
		&\psi^T(J-I)\psi+(\bone-\psi)^T(J-I)(\bone-\psi)	\\
		= & \left( \sum_{i}\psi_i \right)^2 + \left(n-\sum_{i}\psi_i\right)^2-\psi^T\psi -(1-\psi)^T(1-\psi)	\\
		\geq & n^2/2-2n, 
	}
	since the first two terms are minimized at $\sum_{i}\psi_i = n/2$.
	
	The result for $q^{(1)}$ is proved analogously.
\end{proof}	

\begin{proof}[Proof of Proposition~\ref{prop:stationary_characterization}]
Let $\psi= \zeta_1u_1+\zeta_2u_2+w$, $w\in \spn\{u_1,u_2\}^\perp$, be a stationary point. We will consider the population version of all the updates and replace $A$ with $\E(A|Z) := \tilde{P}$ and $\rho_n\to 0$. By Lemma~\ref{lem:unknownpupdate}, 
	\begin{align}
	\tilde{p}&=\frac{\pstar+\qstar}{2}+\underbrace{\frac{(\pstar-\qstar)(\zeta_2^2-x/2n^2)}{\zeta_1^2+(1-\zeta_1)^2-x/n^2}}_{\epsilon'_1}, 	\nonumber\\
	\tilde{q}&=\frac{\pstar+\qstar}{2}-\underbrace{\frac{(\pstar-\qstar)(\zeta_2^2+y/2n^2)}{2\zeta_1(1-\zeta_1)-y/n^2}}_{\epsilon'_2}.
	\label{eq:pqupdate_pop}
	\end{align}
In this case, the update equation~\eqref{eq:bcavi_known_pq} becomes
\begin{align}
\xi & = 4\tilde{t} (\tilde{P}- \tilde{\lambda} (J - I))(\psi^{(s)} - \frac{1}{2}\bone)	\notag\\
 & = 4\tilde{t} n \left( \left(\zeta_1-\frac{1}{2}\right)\left(\frac{\pstar+\qstar}{2}-\tilde{\lambda}\right)u_1 + \frac{\pstar-\qstar}{2}\zeta_2u_2 \right)+ 4\tilde{t}(\tilde{\lambda}-\pstar)\left(\psi-\frac{1}{2}\bone\right)		\notag\\
  & := n\tilde{a} + \tilde{b}
 \label{eq:bcavi_stationary}
 \end{align}
where $\tilde{\lambda}$ and $\tilde{t}$ are defined in terms of $\tilde{p}$ and $\tilde{q}$. Since $\psi$ is a stationary point, the above update gives $\psi=g(\xi)$. 

We consider the following cases.

\textbf{Case 1:}  $\zeta_2^2=\Omega(1)$. Since $\zeta_1(1-\zeta_1)\geq\zeta_2^2$, it is easy to see that ~\eqref{eq:pqupdate_pop} implies that $\tilde{p} > \frac{\pstar+\qstar}{2} > \tilde{q}$, thus $\tilde{p}-\tilde{q} = \Omega(\rho_n)$, $\tilde{t}=\Omega(1)$, $\tilde{p} < \tilde{\lambda} < \tilde{q}$. It follows then $\tilde{b}_i=O(\rho_n)$, and $|\tilde{a}_i|=\Omega(\rho_n)$ for $i\in\cC_1$ or $i\in\cC_2$ (or both). In any of these cases, $\| w \| = O(\rho_n\sqrt{n})=o(\sqrt{n})$.

\textbf{Case 2:} $\zeta_2=o(1)$. Note that $\psi^T(\bone-\psi) \geq 0$ implies that $\zeta_1(1-\zeta_1) - \frac{\| w \|^2}{n} \geq \zeta_2^2$. If $\| w \|^2 = o(n)$, we are done. If $\| w \|^2 = \Omega(n)$,  $\zeta_1(1-\zeta_1) =\Omega(1)$. In this case, $\tilde{p}=\frac{\pstar+\qstar}{2}+ O(\rho_n\zeta_2^2)$, and similarly for $\tilde{q}$. It follows then that $\tilde{t}=O(\zeta_2^2)=o(1)$, $\tilde{\lambda}=\frac{\pstar+\qstar}{2}+o(\rho_n)$ (we defer the details to~\eqref{eq:log_ratio_bound}-~\eqref{eq:lambdalb}). Also note that $\tilde{b}_i=O(\rho_n \zeta_2^2)$. When $n|\tilde{a}_i| \gg \tilde{b}_i$, $g(\xi_i) = g(n\tilde{a}_i) + o(1)$. Since $g(n\tilde{a})\in \text{span}\{u_1,u_2\}$, this implies that $\| w \| = o(\sqrt{n})$. When $n|\tilde{a}_i| \asymp \tilde{b}_i$, $\xi_i=o(1)$, and so we have $\| w \| = o(\sqrt{n})$ again.   
 \end{proof}

\begin{proof}[Proof of Lemma~\ref{lem:iidpsi-unknown}]
	Let $\astar=(\pstar+\qstar)/2$. By \eqref{eq:main_decomp}, define $\kappa_1:=4t^{(1)}\left(\zeta_1-\frac{1}{2}\right)(\astar-\lambda^{(1)})$ and $\kappa_2=4t^{(1)}\zeta_2\frac{\pstar-\qstar}{2}$.
	Consider the initial distribution $\psi^{(0)}_i\stackrel{iid}{\sim} f_\mu$, where $f$ is a distribution supported on $(0,1)$ with mean $\mu$.
	Note that we have the following:
	\begin{align}\label{eq:zeta-eqmean-rand}
	\zeta_1&=\frac{(\psi^{(0)})^T \bone}{n}=\mu+O_P(1/\sqrt{n}),\\
	\zeta_2&=\frac{(\psi^{(0)})^T u_2}{n}=O_P(1/\sqrt{n}).\nonumber
	\end{align} 
	
	Now using \eqref{eq:pqupdate}, recall that
		\begin{align}
	\phat&=\frac{\pstar+\qstar}{2}+\underbrace{\underbrace{\frac{(\pstar-\qstar)(\zeta_2^2-x/2n^2)}{\zeta_1^2+(1-\zeta_1)^2-x/n^2}}_{\epsilon'_1} + O_P(\sqrt{\rho_n}/n)}_{\epsilon_1},	\nonumber\\
	\qhat&=\frac{\pstar+\qstar}{2}-\underbrace{\underbrace{\frac{(\pstar-\qstar)(\zeta_2^2+y/2n^2)}{2\zeta_1(1-\zeta_1)-y/n^2}}_{\epsilon'_2} - O_P(\sqrt{\rho_n}/n)}_{\epsilon_2}.
	\label{eq:pqupdateapp}
	\end{align}
	This gives
	\begin{align*}
	\epsilon_1&= \epsilon'_1 + O_P\left(\frac{\sqrt{\rho_n}}{n}\right)=O_P\left(\frac{\rho_n}{n}\right)+O_P\left(\frac{\sqrt{\rho_n}}{n}\right)=O_P\left(\frac{\sqrt{\rho_n}}{n}\right), \\
	\epsilon_2&=\epsilon'_2+ O_P\left(\frac{\sqrt{\rho_n}}{n}\right)=O_P\left(\frac{\sqrt{\rho_n}}{n}\right).
	\end{align*}
	
	We will use the following logarithmic inequalities for $a >\epsilon >0$:
	\begin{align}\label{eq:log_ratio_bound}
	\frac{2\epsilon}{a+\epsilon}\leq \log\frac{a+\epsilon}{a-\epsilon}\leq \frac{2\epsilon}{a-\epsilon}.
	\end{align}
	Now we have 
	\begin{align}
	t^{(1)}&=\frac{1}{2}\left(\log\left(\frac{\astar+\epsilon_1}{\astar-\epsilon_2}\right)+\log\left(\frac{1-\astar+\epsilon_2}{1-\astar-\epsilon_1}\right)\right)\nonumber,\\
	2t^{(1)}&\geq  \frac{\epsilon_1+\epsilon_2}{\astar+\epsilon_1}+\frac{\epsilon_1+\epsilon_2}{1-\astar+\epsilon_2}\geq\frac{(\epsilon_1+\epsilon_2)}{(\astar+\epsilon_1)(1-\astar+\epsilon_2)}\nonumber,\\
	2t^{(1)}&\leq \frac{(\epsilon_1+\epsilon_2)}{(\astar-\epsilon_2)(1-\astar-\epsilon_1)}.\label{eq:t}
	\end{align}
	For $\lambda^{(1)}$, if $\epsilon_1+\epsilon_2\geq 0$, we have
	\begin{align}
	\lambda^{(1)}&=\frac{\log \frac{1-\qhat}{1-\phat}}{\log\frac{\phat}{\qhat}+\log \frac{1-\qhat}{1-\phat}}\leq 
	\left.\frac{\epsilon_1+\epsilon_2}{1-\astar-\epsilon_1}\right/\left(\frac{\epsilon_1+\epsilon_2}{\astar+\epsilon_1}+\frac{\epsilon_1+\epsilon_2}{1-\astar-\epsilon_1}\right)=\astar+\epsilon_1.\\
	\lambda^{(1)}&\geq \left.\frac{\epsilon_1+\epsilon_2}{1-\astar+\epsilon_2}\right/\left(\frac{\epsilon_1+\epsilon_2}{\astar-\epsilon_2}+\frac{\epsilon_1+\epsilon_2}{1-\astar+\epsilon_2}\right)=\astar-\epsilon_2.
	\label{eq:lambdaub}
	\end{align}
	
	If $\epsilon_1+\epsilon_2\leq 0$,
	\begin{align}\label{eq:lambdalb}
	\lambda^{(1)}&=\frac{\log \frac{1-\qhat}{1-\phat}}{\log\frac{\phat}{\qhat}+\log \frac{1-\qhat}{1-\phat}}\geq 
	\left.\frac{\epsilon_1+\epsilon_2}{1-\astar-\epsilon_1}\right/\left(\frac{\epsilon_1+\epsilon_2}{\astar+\epsilon_1}+\frac{\epsilon_1+\epsilon_2}{1-\astar-\epsilon_1}\right)= \astar+\epsilon_1,\\
	\lambda^{(1)}&\leq \left.\frac{\epsilon_1+\epsilon_2}{1-\astar+\epsilon_2}\right/\left(\frac{\epsilon_1+\epsilon_2}{\astar-\epsilon_2}+\frac{\epsilon_1+\epsilon_2}{1-\astar+\epsilon_2}\right)=\astar-\epsilon_2.\nonumber
	\end{align}
The above analysis shows $t^{(1)}=O_P(\frac{1}{n\sqrt{\rho_n}})$, $|a-\lambda^{(1)}|=O_P(\frac{\sqrt{\rho_n}}{n})$.

We next try to generalize the above calculations for any iteration $s$. For convenience we assume $A$ has self loops, which makes no difference to the asymptotics. Note that, for some $|\xi'|<\xi$, since $g''(0)=0$, 
\begin{align*}
\psi=g(\xi)=\frac{1}{2}+\frac{1}{4}\xi+g'''(\xi') \frac{\xi^3}{3!}=\frac{1}{2}+\frac{1}{4}\xi+O(\xi^3)
\end{align*}
using  the fact that $g'''(\xi)=O(1)$ $\forall\xi$. Substituting, we have:
\ba{
\zeta_1^{(s)}&=\frac{1}{n}\left\langle\psi^{(s)},\bone\right\rangle
=\frac{1}{2}+\frac{1}{4n}\left\langle\xi^{(s)},\bone\right\rangle + O\left(\frac{\|(\xi^{(s)})^3\|_2}{\sqrt{n}}\right)	\nonumber\\
&=\frac{1}{2}+\frac{t^{(s)}}{n}\left\langle (A-\lambda^{(s)} J)(\psi^{(s-1)}-\frac{1}{2}\bone),\bone\right\rangle+O\left(\frac{\|(\xi^{(s)})^3\|_2}{\sqrt{n}}\right), 
}
using the update equation for $\xi^{(s)}$ in~\eqref{eq:bcavi_unknown_pq} and assuming $A$ has self loops for convenience. Here using the decomposition $A=P+(A-P)$,
\ba{
\left\langle (P-\lambda^{(s)}J)(\psi^{(s-1)}-\frac{1}{2}\bone), \bone \right\rangle & =  n^2\left(\frac{\pstar+\qstar}{2}-\lambda^{(s)}\right)(\zeta_1^{(s-1)}-1/2)	\nonumber\\
\left\langle (A-P)(\psi^{(s-1)}-\frac{1}{2}\bone), \bone \right\rangle & \leq \sqrt{n}\|A-P\|_{op} \|\psi^{(s-1)}-\frac{1}{2}\bone\|_2	\nonumber\\
& = O_P(\sqrt{n^2\rho_n}) \|\psi^{(s-1)}-\frac{1}{2}\bone\|_2,	\nonumber
}
where the first line follows from $P-\lambda^{(s)}J = \left(\frac{\pstar+\qstar}{2}-\lambda^{(s)}\right) \bone\bone^T + \frac{\pstar-\qstar}{2} u_2u_2^T$.
It follows then
\ba{
& |\zeta_1^{(s)}-\frac{1}{2}|		\nonumber\\
\leq & \frac{4|t^{(s)}|}{n}\left(n^2\left(\frac{\pstar+\qstar}{2}-\lambda^{(s)}\right) |\zeta_1^{(s-1)}-1/2|+O_P(\sqrt{n^2\rho_n})\|\psi^{(s-1)}-\frac{1}{2}\bone\|_2  \right) +O\left(\frac{\|(\xi^{(s)})^3\|_2}{\sqrt{n}}\right)\nonumber\\
= & 4|t^{(s)}|(n\left(\frac{\pstar+\qstar}{2}-\lambda^{(s)}\right)|\zeta_1^{(s-1)}-1/2|+O_P(\sqrt{\rho_n})\|\psi^{(s-1)}-\frac{1}{2}\bone\|_2)+O\left(\frac{\|\xi^{(s)}\|_2^3}{\sqrt{n}}\right)
\label{eq:zeta1_s}
}
since $\|v^3\|_2=\sqrt{\sum_i v_i^6}\leq \|v\|_2\|v\|_\infty^2\leq \|v\|_2^3$ for any $v$. Similarly, we have:
\ba{
	\zeta_2^{(s)}&=\frac{1}{n}\left\langle\psi^{(s)},u_2\right\rangle
	=\frac{1}{4n}\left\langle\xi^{(s)},u_2\right\rangle + O\left(\frac{\|(\xi^{(s)})^3\|_2}{\sqrt{n}}\right),\nonumber\\
	&=\frac{t^{(s)}}{n}\left\langle (A-\lambda^{(s)} J)(\psi^{(s-1)}-\frac{1}{2}\bone),u_2\right\rangle+O\left(\frac{\|(\xi^{(s)})^3\|_2}{\sqrt{n}}\right),\\
	|\zeta_2^{(s)}|&\leq \frac{|t^{(s)}|}{n}\left(\frac{n^2(\pstar-\qstar)}{2}|\zeta_2^{(s-1)}|+O_P(\sqrt{n^2\rho_n})\|\psi^{(s-1)}-\frac{1}{2}\bone\|_2\right) +O\left(\frac{\|(\xi^{(s)})^3\|_2}{\sqrt{n}}\right)\nonumber\\
	&=|t^{(s)}|(O(n\rho_n)|\zeta_2^{(s-1)}|+O_P(\sqrt{\rho_n})\|\psi^{(s-1)}-\frac{1}{2}\bone\|_2)+O\left(\frac{\|\xi^{(s)}\|^3_2}{\sqrt{n}}\right)
	\label{eq:zeta2_s}
}

For the norm of $\xi^{(s)}$,
\ba{
\|\xi^{(s)}\|_2&\leq 4|t^{(s)}|\left(n^{3/2}\left(\left|\left(\frac{\pstar+\qstar}{2}-\lambda^{(s)}\right)(\zeta_1^{(s-1)}-1/2)\right|+O(\rho_n)|\zeta_2^{(s-1)}|\right)	\right. \nonumber\\
	& \qquad\qquad	\left.+O_P(\sqrt{n\rho_n})\|\psi^{(s-1)}-\frac{1}{2}\bone\|_2\right)\nonumber\\
}
using the same eigen-decomposition on $P$.

To bound $t^{(s)}$, we can first define $\epsilon_1^{(s)}$ and $\epsilon_2^{(s)}$ in the same way as~\eqref{eq:pqupdateapp}, where the order terms come from the second part of~\eqref{eq:p1_update} (and an analogous equation for $q^{(1)}$, with general $\psi^{(s-1)}$ replacing $\psi$). Then provided $\zeta_1^{(s-1)}$ is bounded away from 0 and 1, and $\epsilon_1^{(s)}, \epsilon_2^{(s)} = o_P(\rho_n)$, by~\eqref{eq:t},
\ba{
|t^{(s)}|&= O_P(\zeta_2^{(s-1)})^2 + O(\frac{1}{n^2\rho_n})\left( (\psi^{(s-1)})^T(A-P)\psi^{(s-1)} + (\bone-\psi^{(s-1)})^T(A-P)(\bone-\psi^{(s-1)}) \right)	\nonumber\\
& \qquad \qquad + O_P(\frac{1}{n^2\rho_n})(\psi^{(s-1)})^T(A-P)(\bone-\psi^{(s-1)}),
\label{eq:xi_s}
}
where for any $\psi$,
\bas{
	\psi^T(A-P)\psi&=\frac{1}{4}\bone^T(A-P)\bone+\bone^T(A-P)(\psi-\frac{1}{2}\bone)+(\psi-\frac{1}{2}\bone)^T(A-P)(\psi-\frac{1}{2}\bone)\\
	&=O_P(\sqrt{n^2\rho_n})(1+\|\psi-\frac{1}{2}\bone\|_2)+O_P(\sqrt{n\rho_n})\|\psi-\frac{1}{2}\bone\|_2^2	\\
	&= O_P(\sqrt{n^2\rho_n})(1+\|\psi-\frac{1}{2}\bone\|_2)
}
since $\|\psi^{(s-1)}-\frac{1}{2}\bone\|\leq \sqrt{n}$. Similarly
\bas{
	\psi^T(A-P) \bone & = (\psi-\frac{1}{2}\bone)^T(A-P)\bone + \frac{1}{2} \bone^T(A-P)\bone	\\
	& = O_P(\sqrt{n^2\rho_n})(1+\|\psi-\frac{1}{2}\bone\|_2).
}
The upper bound on $t^{(s)}$ becomes:
\begin{align*}
|t^{(s)}|=O_P\left(|\zeta_2^{(s-1)}|^2\right)+O_P\left( \frac{1}{\sqrt{n^2\rho_n}}\right) (1+\|\psi^{(s-1)}-\frac{1}{2}\bone\|_2)
\end{align*}

In a similar way to bound $\lambda^{(s)}$, note that defining general $\epsilon_1^{(s)}, \epsilon_2^{(s)}$ in~\eqref{eq:t}-\eqref{eq:lambdalb}, as long as $\zeta_1^{(s-1)}$ is bounded away from 0 and 1, and $\epsilon_1^{(s)}, \epsilon_2^{(s)} = o_P(\rho_n)$, we have:
\bas{
	\left|\frac{\pstar+\qstar}{2}-\lambda^{(s)}\right|=O_P\left(\rho_n t^{(s)}\right)
}

Finally,
\ba{
\|\psi^{(s)}-\frac{1}{2}\bone\|_2&=\frac{1}{4}\|\xi^{(s)}\|_2+O\left(\|(\xi^{(s)})^3\|_2\right) \nonumber\\
& =\frac{1}{4}\|\xi^{(s)}\|_2 + O\left(\|(\xi^{(s)})\|^3_2\right) 
\label{eq:psi_half_s}
}

For $s=1$, we have the following:
\bas{
	&t^{(1)}=O_P(\frac{1}{n\sqrt{\rho_n}}), \frac{\pstar+\qstar}{2}-\lambda^{(1)}=O_P(\frac{\sqrt{\rho_n}}{n}),\\
	&\|\xi^{(1)}\|_2, 
	\|\psi^{(1)}-\frac{1}{2}\bone\|_2=O_P(1),\\ &|\zeta_1^{(1)}-1/2|,|\zeta_2^{(1)}|=O_P\left(\frac{1}{\sqrt{n}}\right)
}
where the second line follows from~\eqref{eq:xi_s},~\eqref{eq:psi_half_s}, noting $\zeta_1^{(0)}=O_P(1)$, $\zeta_2^{(0)}=O_P(n^{-1/2})$. The last line follows from~\eqref{eq:zeta1_s} and~\eqref{eq:zeta2_s}.

For $s=2$, note that the above bounds imply $\zeta_1^{(1)}$ is bounded away from 0 and 1, and $\epsilon_1^{(s)}, \epsilon_2^{(s)}=o_P(\rho_n)$. Using the same set of equations again, we have:
\bas{	&t^{(2)}=O_P(\frac{1}{n\sqrt{\rho_n}}),\frac{\pstar+\qstar}{2}-\lambda^{(2)}=O_P(\frac{\sqrt{\rho_n}}{n}) \\
&\|\xi^{(2)}\|_2, \|\psi^{(2)}-\frac{1}{2}\bone\|_2
=O_P(\sqrt{\rho_n})\\
&|\zeta_1^{(2)}-1/2|,|\zeta_2^{(2)}|
=O_P\left(\sqrt{\frac{\rho_n}{n}}\right)
}

In general, once $\|\psi^{(s-1)}-\frac{1}{2}\bone\|_2=O_P(1)$,  $|\zeta_1^{(s-1)}-1/2|$ and $|\zeta_2^{(s-1)}|=O_P(1/\sqrt{n})$ we have
$t^{(s)}=O_P(1/n\sqrt{\rho_n})$, $(\pstar+\qstar)/2-\lambda^{(s)}=O_P(\sqrt{\rho_n}/n)$,  $\|\xi^{(s)}\|_2=O_P(\sqrt{\rho_n})$, $|\zeta_1^{(s)}-1/2|,|\zeta_2^{(s)}|$ are both $o_P(1/\sqrt{n})$ and $\|\psi^{(s)}-\frac{1}{2}\bone\|_2=O_P(\sqrt{\rho_n})$. So for all $s\geq 2$, $\|\psi^{(s)}-\frac{1}{2}\bone\|_2=O_P(\sqrt{\rho_n})$.

\end{proof}

\begin{proof}[Proof of Lemma~\ref{lem:goodinitunknown}]
	In this setting, we write $\phat,\qhat$ as follows:
	\begin{align}
	\phat=\pstar-(\pstar-\qstar)\frac{\frac{\zeta_1^2+(1-\zeta_1)^2}{2}-\zeta_2^2}{\zeta_1^2+(1-\zeta_1)^2-x/n^2} + O_P(\sqrt{\rho_n}/n),	\notag\\
	\qhat=\qstar+(\pstar-\qstar)\frac{\zeta_1(1-\zeta_1)-\zeta_2^2-y/n^2}{2\zeta_1(1-\zeta_1)-y/n^2} + O_P(\sqrt{\rho_n}/n).
	\label{eq:pqupdate2}
	\end{align}
	From the proof of Lemma~\ref{lem:iidpsi-unknown}, \eqref{eq:pqupdateapp}, and \eqref{eq:pqupdate2}, we have: $\epsilon_1 , \epsilon_2 <\frac{\pstar+\qstar}{2}$. 

	\bk
	Also note that $\epsilon_1,\epsilon_2= \Omega_P(-(\pstar-\qstar)\zeta_2^2+\sqrt{\rho_n}/n)$.
	Hence, by the same argument as in Lemma~\ref{lem:iidpsi-unknown}, $|(\pstar+\qstar)/2-\lambda^{(1)}| \leq \max (|\epsilon_1|, |\epsilon_2|) = \frac{\pstar-\qstar}{2}+O_P(1/n)$ by~\eqref{eq:pqupdate2}. 
	
	Finally we see that $$t^{(1)}=\Theta(\frac{\epsilon_1+\epsilon_2}{\rho_n})=\Theta\left((\pstar-\qstar)\zeta_2^2/\rho_n\right).$$
	
	In addition, condition~\eqref{eq:museparation} implies that $\zeta_2^2=\Omega_P(1)$,  we see that $t^{(1)}=\Omega_P(1)$ using ~\eqref{eq:t}. 
	
	Next define 
	\begin{align}\label{eq:kappa}
	\kappa_1&=4t^{(1)}(\zeta_1-\frac{1}{2})(\astar-\lambda^{(1)}) \nonumber\\
	\kappa_2&=4t^{(1)}\zeta_2\frac{(\pstar-\qstar)}{2}.
	\end{align}
	Using \eqref{eq:zeta-eqmean} and ~\eqref{eq:kappa}, 
	\begin{align*}
	\kappa_1 + \kappa_2  &= 4t^{(1)}\left( \frac{\mu_1+\mu_2-1}{2}\left(\frac{\pstar+\qstar}{2}-\lambda^{(1)}\right) + \frac{(\mu_1-\mu_2)(\pstar-\qstar)}{4}+O_P(\rho_n/\sqrt{n})\right),\\
	\kappa_1 - \kappa_2 & = 4t^{(1)}\left( \frac{\mu_1+\mu_2-1}{2}\left(\frac{\pstar+\qstar}{2}-\lambda^{(1)}\right) - \frac{(\mu_1-\mu_2)(\pstar-\qstar)}{4} + O_P(\rho_n/\sqrt{n})\right).
	\end{align*}
	
	From \eqref{eq:main_decomp} and adding the noise term from the sample version of the update, 
	\begin{align}
	\xi_i^{(1)}=n(\kappa_1+\sigma_i\kappa_2)+b_i^{(0)} + nr_i^{(0)},
	\label{eq:kappa_r}
	\end{align}
	In~\eqref{eq:kappa_r}, $b_i^{(0)}$ is of smaller order than the other terms and it suffices to consider $n(\kappa_1+\sigma_i\kappa_2+r_i^{(0)})$. 
	Since $|r_i^{(0)}|=O_P\left(\sqrt{\frac{\rho_n\log^2 n}{n}}\right)$ (see proof of Theorem~\ref{thm:samp_known_pq}),  for any  pair $i\in C_1$ and $j\in C_2$ we have 
	\begin{align*}
	(\kappa_1+\kappa_2+r_i^{(0)})&(\kappa_1-\kappa_2+r_j^{(0)})\\
	\leq &(\kappa_1^2-\kappa_2^2)+ O\left(\max(|r_i^{(0)}|,|r_j^{(0)}|)\max(|\kappa_1|,|\kappa_2|)\right)\\
	= & (\kappa_1^2-\kappa_2^2)+O_P\left((\pstar-\qstar)\sqrt{\frac{\rho_n\log^2 n}{n}}\right)\\
	= & (t^{(1)})^2(\pstar-\qstar)^2\left( (\mu_1+\mu_2-1)^2 - (\mu_1-\mu_2)^2 +O_P\left(\frac{1}{\pstar-\qstar}\sqrt{\frac{\rho_n\log^2 n}{n}}\right) \right)<0.
	\end{align*}
	Thus $n(\kappa_1+\kappa_2+r_i^{(0)})$ and $n(\kappa_1-\kappa_2+r_j^{(0)})$, for $i,j$ in different blocks, have opposite signs. We will now check if $n(\kappa_1+\sigma_i\kappa_2+r_i^{(0)})\rightarrow \infty$, and it suffices to lower bound $n(|\kappa_2|-|\kappa_1|-\max_i|r_i^{(0)}|)$.  Since $|\mu_1-\mu_2|\geq 2|\mu_1+\mu_2-1|+O_P\left(\frac{\sqrt{\rho_n\log^2 n/n}}{\pstar-\qstar}\right)$, 
	\begin{align*}
	 n(|\kappa_2|-|\kappa_1|-\max_i|r_i^{(0)}|)
	 &\geq nt^{(1)}\left(|\mu_1-\mu_2|(\pstar-\qstar)-|\mu_1+\mu_2-1|(\pstar-\qstar)-O_P\left(\sqrt{\frac{\rho_n\log^2 n}{n}}\right)\right)\\
	 &\geq nct^{(1)}(\pstar-\qstar) |\mu_1-\mu_2|=\Theta\left(|\mu_1-\mu_2|^3 n\frac{(\pstar-\qstar)^2}{\rho_n}\right),
	\end{align*} 
	for some constant $c$, so as long as $|\mu_1-\mu_2|\geq \left(\frac{\rho_n\log n}{n(\pstar-\qstar)^2}\right)^{1/3}$.
	
	Thus $\kappa_1+\sigma_i\kappa_2 + r_i^{(0)}$ is growing to infinity with an order bounded below by $\Omega_P(\log n)$.
	

	If $n(\kappa_1+\kappa_2 + r_i^{(0)})>0$, since $\psi^{(1)}_i= g(n(\kappa_1+\sigma_i\kappa_2)+b_i^{(0)}  + nr_i^{(0)})$, we have $\psi^{(1)}= \icone+O_P(\exp(-\Omega(\log n)))$. The case $\kappa_1+\kappa_2+ r_i^{(0)}<0$ is similar.
	
\end{proof}

\bibliographystyle{plain}

\end{document}